\documentclass[11pt]{article}

\usepackage{graphicx}
\usepackage{caption}
\usepackage{subcaption}

\usepackage[margin = 1in]{geometry}
\usepackage{amsmath,amssymb,amsfonts,amsthm}
\usepackage[retainorgcmds]{IEEEtrantools}
\usepackage{color}
\usepackage{hyperref} 
\def \p{\partial}
\def \ds{\displaystyle}
\def \grad{\nabla}

\newcommand{\bey}{\begin{eqnarray}}
\newcommand{\eey}{\end{eqnarray}}

\newcommand{\beq}{\begin{equation}}
\newcommand{\eeq}{\end{equation}}
\theoremstyle{plain}
\newtheorem{thm}{\hspace{6mm}Theorem}[section]

\newtheorem{lem}{\hspace{6mm}Lemma}[section]

\theoremstyle{definition}

\theoremstyle{remark}
\newtheorem{exam}{\hspace{6mm}Example}[section]
\newtheorem{rem}{\hspace{6mm}Remark}[section]

\newenvironment{keywords}%
   {\begin{trivlist}\item[]{\bfseries\sffamily Key words:}~}
   {\end{trivlist}}
\newenvironment{AMS}%
   {\begin{trivlist}\item[]{\bfseries\sffamily AMS subject classifications:}~}
   {\end{trivlist}}

\allowdisplaybreaks

\title{%
Monotone finite difference schemes for anisotropic diffusion problems via nonnegative directional splittings%
\thanks{%
      This work was supported in~part by the National Science Foundation through grant DMS-1115118.%
   }%
}

\author{%
   Cuong~Ngo%
   \thanks{%
      Department of~Mathematics, University of~Kansas, Lawrence, KS~66045, U.S.A.
      (\href{mailto:cngo@ku.edu}{\nolinkurl{cngo@ku.edu}}).%
   }%
   \and
   Weizhang~Huang%
   \thanks{%
      Department of~Mathematics, University of~Kansas, Lawrence, KS~66045, U.S.A.
      (\href{mailto:whuang@ku.edu}{\nolinkurl{whuang@ku.edu}}).%
   }%
 }

\date{}

\begin{document}
\maketitle

\begin{abstract}
Nonnegative directional splittings of anisotropic diffusion operators in the divergence form are investigated.
Conditions are established for nonnegative directional splittings to hold in a neighborhood of an arbitrary
interior point. The result is used
to construct monotone finite difference schemes for the boundary value problem of anisotropic diffusion operators.
It is shown that such a monotone scheme can be constructed if the underlying diffusion matrix
is continuous on the closure of the physical domain and symmetric and uniformly positive definite on the domain,
the mesh spacing is sufficiently small, and the size of finite difference stencil is sufficiently large.
An upper bound for the stencil size is obtained, which is determined completely by the diffusion matrix.
Loosely speaking, the more anisotropic the diffusion matrix is, the larger stencil is required.
An exception is the situation with a strictly diagonally dominant diffusion matrix where a three-by-three stencil
is sufficient for the construction of a monotone finite difference scheme.
Numerical examples are presented to illustrate the theoretical findings.
\end{abstract}



\section{Introduction}
We consider the finite difference solution of the boundary value problem
\begin{equation} \label{had}
	\begin{cases}	
		-\nabla \cdot (\mathbb{D} \nabla u) = f, & \quad \text{in} \quad \Omega \\
		u = g, & \quad \text{on} \quad \p \Omega
	\end{cases}
\end{equation}
where $\Omega = (0,1)\times (0,1)$,
$f$ and $g$ are given functions, and the diffusion matrix,
\begin{equation} \label{eqn:D}
\mathbb{D} =\begin{bmatrix} a(x,y) & b(x,y) \\ b(x,y) & c(x,y) \end{bmatrix},
\end{equation}
is assumed to be continuous on $\overline{\Omega} \equiv \Omega \cup \partial \Omega$ and symmetric and
uniformly positive definite on $\Omega$. We are interested in the case where $\mathbb{D}$
depends on spatial location (heterogeneous diffusion) and has unequal eigenvalues at least on some portion
of $\Omega$ (anisotropic diffusion). For this case, (\ref{had}) is often called a heterogeneous anisotropic
diffusion problem. Anisotropic diffusion arises from various areas of science and engineering, including
plasma physics~\cite{GL09}, petroleum reservoir simulation~\cite{EAK01,MD06}, and image
processing~\cite{KM09,Weickert}. 

When a conventional method such as a finite element, a finite difference, or a finite volume method, is applied
to this problem, spurious oscillations can occur in the numerical solution. This difficulty can be
overcome by numerical schemes satisfying a discrete maximum principle (DMP). The development
and studies of DMP-preserving schemes have received considerable attention in the past
particularly in the context of finite element and finite volume methods, e.g., see
\cite{Ciarlet1970, CiarletRaviart1973,
Crumpton1995,Draganescu, GL09, Gunter2007, Gunter2005, Kuzmin2009, Huang1, Huang2, Li2007,
Lipnikov2007, Liska2008, Mlacnik2006, LePotier2005, LePotier2009-1, LePotier2009, Sharma2007,Stoyan1982,
Stoyan1986, Varga, YS08, Shu13}.
On the other hand, relatively less work has been done in the context of finite difference (FD) methods.
The major effort has been made to study monotone (or called positive-type or nonnegative-type) schemes,
which belong to a special class of DMP-preserving schemes with the coefficient matrix (or the Jacobian matrix
in the nonlinear case) of the corresponding algebraic system being an $M$-matrix (e.g., see \cite{Varga}
for the definition of $M$-matrices).
For example, Motzkin and Wasow \cite{Motzkin1953} prove that a monotone FD scheme exists
for any linear second-order elliptic problem when the mesh is sufficiently fine.
Greenspan and Jain \cite{Greenspan1965}, using nonnegative directional splittings (see the definition below),
propose a way to construct such schemes for elliptic operators in the nondivergence form
\begin{equation}
L[u] \equiv a(x,y) u_{xx} + 2 b(x,y) u_{xy} + c(x,y) u_{yy} \quad \text{ with }\quad b(x,y)^2 < a(x,y) c(x,y) .
\label{GJ-1}
\end{equation}
The results are extended to
elliptic problems in the divergence form (\ref{had}) by Weickert \cite{Weickert}.
Consistent and stable monotone FD schemes are shown to be convergent (to the solution) for
linear second-order elliptic problems by Bramble et al. \cite{Bramble1969} and (to the viscosity solution)
for nonlinear second-order degenerate elliptic or parabolic partial differential equations
by Barles and Souganidis \cite{Barles}. Oberman \cite{Oberman} studies degenerate elliptic schemes
(a special type of monotone scheme) for a general class of nonlinear degenerate elliptic problems.

The objective of this work is to study the construction of monotone FD schemes for problems in the form of (\ref{had})
using nonnegative directional splittings. Weickert's results are improved in two aspects.
We first present a condition under which nonnegative directional splittings hold for a neighborhood
of an arbitrary point. As we will see below, it is necessary to consider nonnegative directional splittings in a neighborhood
for the construction of monotone FD schemes.
We then extend the result to the situation where the coefficient $b(x,y)$ can change sign over the domain.
To be more specific, we recall that Weickert  \cite{Weickert} considers the directional splitting
\begin{equation}
    \nabla \cdot (\mathbb{D} \nabla u)(x_0,y_0) = \p_x (\gamma_0 \p_x u)(x_0,y_0)
    + \p_{\beta} (\gamma_1 \p_{\beta} u)(x_0,y_0) + \p_y (\gamma_2 \p_y u)(x_0,y_0),
\label{eqn:split0}
\end{equation}
where $(x_0,y_0)$ is a given point in $\Omega$,
$\p_{\beta} = (\sin(\beta), \cos(\beta))^T \cdot \nabla$, $\gamma_i = \gamma_i(x,y)$ ($i=0, 1, 2$) are functions,
and $\beta$ is a constant (angle) (but can vary with $(x_0,y_0)$). Weickert \cite[Page 90]{Weickert} shows that 
the condition
\begin{equation}
a(x_0,y_0)-b(x_0, y_0) \cot(\beta) \ge 0\quad \text{ and } \quad c(x_0,y_0)-b(x_0,y_0) \tan(\beta) \ge 0 
\label{weickert-1}
\end{equation}
is sufficient to guarantee that $\gamma_i$ ($i=0,1,2$) are nonnegative at $(x_0,y_0)$
(nonnegative directional splitting). Unfortunately, such a point-wise result is insufficient for the construction of
monotone FD schemes for the divergence form (\ref{eqn:split0}) since an FD discretization 
has to use the information of the coefficients $\gamma_i$
in a neighborhood of any mesh point. To avoid this difficulty, we develop a condition under which nonnegative
directional splittings hold in a neighborhood of a given point. Moreover, we study how the condition can
be extended to the situation where $b(x,y)$ changes sign over the domain.
An upper bound for the stencil size is obtained, which is completely determined by the diffusion matrix.
Loosely speaking, the more anisotropic the diffusion matrix is, the larger stencil is required.
An exception is the situation with a strictly diagonally dominant diffusion matrix where a three-by-three stencil
is sufficient for the construction of a monotone finite difference scheme (cf. Theorem~\ref{thm3.2}).

An outline of this paper is as follows. Nonnegative directional splittings for (\ref{had}) are studied in \S\ref{sec:splitting}.
Construction of monotone FD schemes using the results in \S\ref{sec:splitting} is discussed in \S\ref{sec:FD}.
Numerical examples are presented in \S\ref{sec:num}, followed by the conclusions in \S\ref{sec:conclude}.

\section{Nonnegative directional splitting} \label{sec:splitting}
\label{SEC:splitting}

In this section we consider nonnegative directional splittings for problems in the form of (\ref{had}).
Recall that our goal is to use the splittings to construct monotone FD schemes. Thus we are interested
in the situations where such a splitting holds for a neighborhood of a given point. We first consider
the case where $b(x,y)$ does not change sign and then extend the result to the sign changing situation.
Finally, we establish a condition under which the nonnegative directional
splitting holds uniformly over the domain for any neighborhood of a fixed size.
The result is needed in constructing monotone FD schemes on a uniform mesh.

\begin{lem}
\label{lem1}
Consider an arbitrary point $(x_0,y_0)$ in $\Omega$ and a neighborhood $\Omega_0$ of $(x_0,y_0)$. We assume
that $b(x,y)$ does not change sign on $\Omega_0$. If there exists a constant $\beta$ satisfying
\begin{equation}
\begin{cases}
\ds \sup_{\Omega_0} \dfrac{b(x,y)}{a(x,y)} \leq \tan\beta \leq \inf_{\Omega_0} \dfrac{c(x,y)}{b(x,y)},
&\quad \text{for } b \geq 0 \text{ and } b \not \equiv 0 \text{ on } \Omega_0   \\
\ds \sup_{\Omega_0} \dfrac{c(x,y)}{b(x,y)} \leq \tan\beta \leq \inf_{\Omega_0} \dfrac{b(x,y)}{a(x,y)} ,
&\quad \text{for } b \leq 0 \text{ and } b \not \equiv 0 \text{ on } \Omega_0  \\
\beta \text{ is any finite value with } \sin(\beta) \cos(\beta) \neq 0, &\quad \text{for } b \equiv 0 \text{ on } \Omega_0
\end{cases}
\label{lem1:angCond1}
\end{equation}
then we have the nonnegative directional splitting 
\begin{align} \label{lem1:identity}
  \grad \cdot \left( \mathbb{D} \, \grad u \right)(x_0,y_0) & = \p_x(\gamma_0 \, \p_x u)(x_0,y_0) + \p_{\beta}(\gamma_1 \, \p_{\beta} u)(x_0,y_0) + \p_y(\gamma_2 \, \p_y u)(x_0,y_0) ,
\end{align}
where the coefficients $\gamma_0$, $\gamma_1$, and $\gamma_2$, all nonnegative on $\Omega_0$, are given by
\begin{equation}
\label{lem1:coef}
\begin{cases}
\gamma_0(x,y) = a(x,y) - b(x,y) \cot\beta \geq 0 ,
\\
\gamma_1(x,y) = \dfrac{b(x,y)}{\cos\beta \sin\beta} \geq 0 ,  \\
\gamma_2(x,y) = c(x,y) - b(x,y) \tan\beta \geq 0 .
\end{cases}
\end{equation}
\end{lem}

\begin{proof}
For the situation with $b(x,y) \equiv 0$ on $\Omega_0$, it is obvious that (\ref{lem1:identity}) and (\ref{lem1:coef})
hold with any $\beta$ satisfying $\sin(\beta) \cos(\beta) \neq 0$.
For other situations, condition (\ref{lem1:angCond1}) implies that
$0 < \beta < \frac{\pi}{2}$ or $-\frac{\pi}{2} < \beta < 0$.
Since $\mathbb{D}(x,y)$ is symmetric and uniformly positive definite on $\Omega_0 \subset \Omega$,
we have $a(x,y) > 0$ and $c(x,y) > 0$ for each $(x,y) \in \Omega_0$.
For each $(x,y) \in \Omega_0$, we denote the eigenvalues of of $\mathbb{D}$ by $\lambda_1$ and $\lambda_2$
(without loss of generality we assume $\lambda_1 \ge \lambda_2$) and the angle formed by the principal
eigenvector and the $x$-axis by $\psi$. Then, the diffusion matrix at $(x,y)$ has the eigen-decomposition as
\begin{align}
\mathbb{D} & = \begin{bmatrix} a & b \\ b & c\end{bmatrix}
= \lambda_1 \begin{bmatrix} \cos(\psi) \\ \sin(\psi) \end{bmatrix} \begin{bmatrix} \cos(\psi) & \sin(\psi) \end{bmatrix}
+ \lambda_2 \begin{bmatrix} -\sin(\psi) \\ \cos(\psi) \end{bmatrix} \begin{bmatrix} -\sin(\psi) & \cos(\psi) \end{bmatrix}
\notag \\
& = \begin{bmatrix} \lambda_1 \cos^2 (\psi) + \lambda_2 \sin^2 (\psi) & (\lambda_1 - \lambda_2) \cos(\psi) \sin(\psi) \\
(\lambda_1 - \lambda_2) \cos(\psi) \sin(\psi) & \lambda_1 \sin^2 (\psi) + \lambda_2 \cos^2 (\psi) \end{bmatrix} .
\label{eqn:eigenDecomp}
\end{align}

In the following, we assume $b \geq 0$ on $\Omega_0$.
(The analysis for the case with $b \leq 0$ is similar.)
In this case, $\psi \in [0, \frac{\pi}{2}]$. We now show that $\grad \cdot (\mathbb{D} \grad u)$ can be written into
the form of (\ref{lem1:identity}) with nonnegative $\gamma_0$, $\gamma_1$, and $\gamma_2$.
By direct calculation, we have
\begin{align*}
\p_{\beta}(\gamma_1 \, \p_{\beta} u) & = \nabla \cdot \left ( \gamma_1 \begin{bmatrix} \cos^2\beta &  \cos\beta \sin\beta\\
\cos\beta \sin\beta & \sin^2\beta \end{bmatrix} \nabla u \right ),
\\
\p_x(\gamma_0 \, \p_x u) & = \nabla \cdot \left ( \gamma_0 \begin{bmatrix} 1 &  0\\
0 & 0 \end{bmatrix} \nabla u \right ),
\\
\p_y(\gamma_2 \, \p_y u) & = \nabla \cdot \left ( \gamma_2 \begin{bmatrix} 0 &  0\\
0 & 1 \end{bmatrix} \nabla u \right ) .
\end{align*}
Then, (\ref{lem1:identity}) is equivalent to
\[
\nabla \cdot \left( \mathbb{D} \, \nabla u \right) = \nabla \cdot \left (
\begin{bmatrix} \gamma_0 + \gamma_1 \cos^2\beta & \gamma_1 \cos\beta \sin\beta \\
\gamma_1 \cos\beta \sin\beta & \gamma_1 \sin^2\beta + \gamma_2 \end{bmatrix} \nabla u \right ),
\]
which requires
\[
\begin{bmatrix} a & b \\ b & c\end{bmatrix} = \begin{bmatrix} \gamma_0 + \gamma_1 \cos^2\beta & \gamma_1 \cos\beta \sin\beta \\
\gamma_1 \cos\beta \sin\beta & \gamma_1 \sin^2\beta + \gamma_2 \end{bmatrix} .
\]
This can be expressed into
\[
\begin{bmatrix} 1 & \cos^2\beta & 0\\ 0 & \cos\beta \sin\beta & 0 \\ 0 & \sin^2\beta & 1\end{bmatrix}
\begin{bmatrix} \gamma_0 \\ \gamma_1 \\ \gamma_2\end{bmatrix} = 
\begin{bmatrix} a \\ b \\ c \end{bmatrix} .
\]
Since $0 < \beta < \frac{\pi}{2}$, we can solve the above equation and get (\ref{lem1:coef}).
The condition (\ref{lem1:angCond1}) implies that $\gamma_0 \geq 0$ and $\gamma_2 \geq 0$ on $\Omega_0$.
From (\ref{eqn:eigenDecomp}), we have
\[
b = (\lambda_1 - \lambda_2) \cos\psi \sin\psi .
\]
Thus,
\begin{equation*} \gamma_1 = \frac{b}{\sin\beta \cos\beta} = \dfrac{(\lambda_1 - \lambda_2) \cos\psi \sin\psi}{\cos\beta \sin\beta} \geq 0\, ,
\end{equation*}
since $\psi \in [0,\frac{\pi}{2}]$ and $\beta \in (0, \frac{\pi}{2})$. 
\end{proof}

\vspace{10pt}

The above lemma shows that for the case where $b$ does not change sign,
(\ref{had}) has a nonnegative directional splitting on $\Omega_0$
if (\ref{lem1:angCond1}) is satisfied. In the following lemma, we consider
the general case where $b$ changes sign on $\Omega_0$. 
Hereafter, we use the superscripts ``$\frac{}{}^+$'' and ``$\frac{}{}^-$''
to indicate the regions associated with $b(x,y)>0$ and $b(x,y)<0$, respectively.
For example, we denote 
\[
\Omega_0^+ = \{ (x,y) \in \Omega_0 : b(x,y) > 0 \},\qquad
\Omega_0^- = \{ (x,y) \in \Omega_0 : b(x,y) < 0 \} .
\]

\vspace{5pt}

\begin{rem}
\label{rem2.1}
The set $\Omega_0^+$ or $\Omega_0^-$ may be empty. 
When this happens, many operations associated with such an empty in the following analysis may not make sense.
For this reason, we assume that the formulas with an empty set will be ignored.
This remark also applies to other empty sets.
\qed
\end{rem}

\begin{lem} 
\label{lem2}
Consider an arbitrary point $(x_0,y_0)$ in $\Omega$ and a neighborhood $\Omega_0$ of $(x_0,y_0)$. If there exist constants $\beta_1$ and $\beta_2$ satisfying
\begin{equation}
\begin{cases}
	\ds \sup_{\Omega_0^+} \dfrac{b(x,y)}{a(x,y)} < \tan\beta_1 < \inf_{\Omega_0^+} \dfrac{c(x,y)}{b(x,y)} , \\
	\ds \sup_{\Omega_0^-} \dfrac{c(x,y)}{b(x,y)} < \tan\beta_2 < \inf_{\Omega_0^-} \dfrac{b(x,y)}{a(x,y)} , 
\end{cases}
\label{lem2:angCond1}
\end{equation}
then we have the nonnegative directional splitting
\begin{align}
\grad \cdot \left( \mathbb{D} \, \grad u \right)(x_0,y_0) & =  \p_x(\gamma_0\, \p_x u)(x_0,y_0) + \p_{\beta_1}(\gamma_1^+ \, \p_{\beta_1} u)(x_0,y_0) \nonumber \\
&  \qquad + \; \p_{\beta_2}(\gamma_1^- \, \p_{\beta_2} u)(x_0,y_0) + \p_y(\gamma_2 \, \p_y u)(x_0,y_0) ,
\label{splitting-1}
\end{align}
where $\gamma_0(x,y)$, $\gamma_1^+(x,y)$, $\gamma_1^-(x,y)$, and $\gamma_2(x,y)$, all nonnegative
on $\Omega_0$, are given by
\begin{align}
\gamma_0(x,y) & = \begin{cases}
a(x,y) - b(x,y)\cot\beta_1,  & \quad \text{for} \quad  b(x,y) \geq 0 \\
a(x,y)- b(x,y) \cot\beta_2 , & \quad \text{for} \quad  b(x,y) < 0  
\end{cases}
\label{gamma0}
\\
\gamma_1^+(x,y) & =
\begin{cases}
\dfrac{b(x,y)}{\cos\beta_1 \sin\beta_1}, & \quad \text{for} \quad  b(x,y) \geq 0 \\
0, & \quad \text{for} \quad  b(x,y) < 0  
\end{cases}
\label{gamma1+}
\\
\gamma_1^-(x,y) & =
\begin{cases}
0 , & \quad \text{for} \quad  b(x,y) \geq 0 \\
\dfrac{b(x,y)}{\cos\beta_2 \sin\beta_2}, & \quad \text{for} \quad  b(x,y) < 0  
\end{cases}
\label{gamma1-}
\\
\gamma_2(x,y) & =
\begin{cases}
c(x,y) - b(x,y)\tan\beta_1 , & \quad \text{for} \quad  b(x,y) \geq 0 \\
c(x,y) - b(x,y) \tan\beta_2 , & \quad \text{for} \quad  b(x,y) < 0 .
\end{cases}
\label{gamma2}
\end{align}
\end{lem}
\begin{proof}
For any given constants $\epsilon > 0$ and $M > 0$ we define a decompose of the diffusion matrix as
\[
\mathbb{D}(x,y) = \mathbb{D}_1(x,y) + \mathbb{D}_2(x,y),
\]
where
\begin{align}
\mathbb{D}_1 & = \begin{bmatrix} a_1 & b_1\\ b_1 & c_1 \end{bmatrix} 
\equiv \begin{cases}\begin{bmatrix}a - \epsilon & b + {\epsilon/M} \\ b + {\epsilon/M} & c - \epsilon \end{bmatrix},
& \quad \text{ for } \quad b(x,y) \geq 0 \\ \\
\begin{bmatrix} \epsilon & 0 \\ 0 & \epsilon \end{bmatrix}, & \quad \text{ for } \quad b(x,y) < 0 
 \end{cases}
\label{def:D1}
\\
\mathbb{D}_2 & = \begin{bmatrix} a_2 & b_2 \\ b_2 & c_2 \end{bmatrix} 
\equiv \begin{cases}\begin{bmatrix} \epsilon & -{\epsilon/M} \\ -{\epsilon/M} & \epsilon \end{bmatrix} ,
& \quad \text{ for } \quad b(x,y) \geq 0 \\ \\
\begin{bmatrix} a-\epsilon & b \\ b & c-\epsilon \end{bmatrix},
& \quad \text{ for } \quad b(x,y) < 0 .
\end{cases}
\label{def:D2}
\end{align}
Then we have
\begin{equation}
\label{eqn:D_splitting} \grad \cdot (\mathbb{D} \grad u) = \grad \cdot (\mathbb{D}_1 \grad u)
+ \grad \cdot (\mathbb{D}_2 \grad u) . 
\end{equation}
Since
\[
b_1  \geq 0 \quad \text{and} \quad b_2 < 0, \quad \forall (x,y) \in \Omega_0 
\]
we can apply Lemma~\ref{lem1} to each term on the right-hand side of (\ref{eqn:D_splitting}).
To this end, we notice that  (\ref{lem2:angCond1}) implies, for sufficiently small $\epsilon$ and large $M$, that
\begin{align}
\label{eMcond1}
\ds \sup_{\Omega_0^+} \dfrac{b}{a} < \sup_{\Omega_0^+} \dfrac{b+{\epsilon/M}}{a-\epsilon} 
< \tan\beta_1 < \inf_{\Omega_0^+} \dfrac{c-\epsilon}{b+{\epsilon/M}} 
< \inf_{\Omega_0^+} \dfrac{c}{b} ,
\\
\label{eMcond2}
\ds \sup_{\Omega_0^-} \dfrac{c}{b} < \sup_{\Omega_0^-} \dfrac{c-\epsilon}{b} 
< \tan\beta_2 < \inf_{\Omega_0^-} \dfrac{b}{a - \epsilon}  < \inf_{\Omega_0^-} \dfrac{b}{a}   .
\end{align}
From (\ref{def:D1}), we have
\[
\sup_{\Omega_0^-} \frac{b_1}{a_1} = 0 \le \sup_{\Omega_0^+} \frac{b_1}{a_1} = \sup_{\Omega_0^+} \frac{b+{\epsilon/M}}{a-\epsilon} ,
\qquad 
\inf_{\Omega_0^-} \frac{c_1}{b_1} = + \infty >  \inf_{\Omega_0^+} \frac{c_1}{b_1} = \inf_{\Omega_0^+} \frac{c-\epsilon}{b+{\epsilon/M}} .
\]
Combining this with (\ref{eMcond1}), we get
\[
\sup_{\Omega_0} \frac{b_1}{a_1} = \sup_{\Omega_0^+} \frac{b_1}{a_1}
< \tan{\beta_1} < \inf_{\Omega_0} \frac{c_1}{b_1} = \inf_{\Omega_0^+} \frac{c_1}{b_1} .
\]
Thus, the condition of Lemma \ref{lem1} is satisfied, which implies that
\begin{equation}
\label{eqn:D1}
\grad \cdot \left( \mathbb{D}_1 \, \grad u \right)
= \p_x(\gamma_{1,0} \, \p_x u) + \p_{\beta_1}(\gamma_{1,1} \, \p_{\beta_1} u) + \p_y(\gamma_{1,2} \, \p_y u) ,
\end{equation}
where $\gamma_{1,0}$, $\gamma_{1,1}$, and $\gamma_{1,2}$ are nonnegative and given by
\begin{align*}
\gamma_{1,0} & = 
\begin{cases}
(a-\epsilon) - (b+{\epsilon/M})\cot(\beta_1),  & \quad \text{for} \quad  b \geq 0 \\
\epsilon,  & \quad \text{for} \quad  b < 0 
\end{cases} \\
\gamma_{1,1} & =
\begin{cases}
\dfrac{b+{\epsilon/M}}{\cos\beta_1 \sin\beta_1},  & \quad \text{for} \quad  b \geq 0 \\
0,  & \quad \text{for} \quad  b < 0  
\end{cases} \\
\gamma_{1,2} & =
\begin{cases}
(c-\epsilon) - (b+{\epsilon/M})\tan\beta_1,  & \quad \text{for} \quad  b \geq 0 \\
\epsilon,  & \quad \text{for} \quad  b < 0  .
\end{cases}
\end{align*}

Similarly, for sufficiently large $M$, from (\ref{def:D2}) we have
\[
\sup_{\Omega_0^+} \frac{c_2}{b_2} = - M \le \sup_{\Omega_0^-} \frac{c_2}{b_2} = \sup_{\Omega_0^-} \dfrac{c-\epsilon}{b} ,
\qquad 
\inf_{\Omega_0^+} \frac{b_2}{a_2} = - \frac{1}{M} \ge \inf_{\Omega_0^-} \frac{b_2}{a_2} = \inf_{\Omega_0^-} \dfrac{b}{a - \epsilon}.
\]
Combining this with (\ref{eMcond2}) gives
\[
\sup_{\Omega_0} \frac{c_2}{b_2} = \sup_{\Omega_0^-} \frac{c_2}{b_2}
< \tan{\beta_2} < \inf_{\Omega_0} \frac{b_2}{a_2} = \inf_{\Omega_0^-} \frac{b_2}{a_2} .
\]
Then from Lemma \ref{lem1} we have
\begin{equation}
\label{eqn:D2}
\grad \cdot \left( \mathbb{D}_2 \, \grad u \right) = \p_x(\gamma_{2,0} \, \p_x u) + \p_{\beta_2}(\gamma_{2,1} \, \p_{\beta_2} u) + \p_y(\gamma_{2,2} \, \p_y u) ,
\end{equation}
where $\gamma_{2,0}$, $\gamma_{2,1}$, and $\gamma_{2,2}$,  all nonnegative, are given by
\begin{align*}
\gamma_{2,0} & = 
\begin{cases}
\epsilon + {\epsilon/M} \cot\beta_2,  & \quad \text{for} \quad  b \geq 0 \\\
(a-\epsilon) - b \cot\beta_2,   & \quad \text{for} \quad  b < 0  
\end{cases} \\
\gamma_{2,1} & =
\begin{cases}
\dfrac{-{\epsilon/M}}{\cos\beta_2 \sin\beta_2},  & \quad \text{for} \quad  b \geq 0 \\
\dfrac{b}{\cos\beta_2 \sin\beta_2},  & \quad \text{for} \quad  b < 0  
\end{cases} \\
\gamma_{2,2} & =
\begin{cases}
\epsilon + {\epsilon/M} \tan\beta_2 , & \quad \text{for} \quad  b \geq 0 \\
(c-\epsilon) - b \tan\beta_2 , & \quad \text{for} \quad  b < 0 .
\end{cases}
\end{align*}

By combining (\ref{eqn:D_splitting}), (\ref{eqn:D1}), and (\ref{eqn:D2}), we get
\[
\grad \cdot (\mathbb{D} \grad u) 
 =  \p_x(\gamma_0\, \p_x u) + \p_{\beta_1}(\gamma_1^+ \, \p_{\beta_1} u) + \p_{\beta_2}
 (\gamma_1^- \, \p_{\beta_2} u) + \p_y(\gamma_2 \, \p_y u) \, ,
\]
where $\gamma_0$, $\gamma_1^+$, $\gamma_1^-$, and $\gamma_2$, all nonnegative, are given by
\begin{align*}
\gamma_0 = \gamma_{1,0} + \gamma_{2,0} & = 
\begin{cases}
a - (b+{\epsilon/M})\cot\beta_1 + {\epsilon/M} \cot\beta_2 , & \quad \text{for} \quad  b \geq 0 \\
a - b \cot\beta_2 ,  & \quad \text{for} \quad  b < 0  
\end{cases} \\
\gamma_2 = \gamma_{1,2} + \gamma_{2,2} & =
\begin{cases}
c - (b+{\epsilon/M})\tan\beta_1 + {\epsilon/M} \tan\beta_2 , & \quad \text{for} \quad  b \geq 0 \\
c - b \tan\beta_2 , & \quad \text{for} \quad  b < 0  
\end{cases} \\
\gamma_1^+ = \gamma_{1,1} & =
\begin{cases}
\dfrac{b+{\epsilon/M}}{\cos\beta_1 \sin\beta_1}, & \quad \text{for} \quad  b \geq 0 \\
0,  & \quad \text{for} \quad  b < 0  
\end{cases} \\
\gamma_1^- = \gamma_{2,1} & =
\begin{cases}
\dfrac{-{\epsilon/M}}{\cos\beta_2 \sin\beta_2},  & \quad \text{for} \quad  b \geq 0 \\
\dfrac{b}{\cos\beta_2 \sin\beta_2},  & \quad \text{for} \quad  b < 0 .
\end{cases}
\end{align*}
Since the above result holds for any sufficiently small $\epsilon$, we take the limit as 
$\epsilon \to 0$ and obtain the conclusion of the lemma.
\end{proof}

\vspace{5pt}

In the next lemma, we show that (\ref{lem2:angCond1}) can be satisfied uniformly over $\Omega$
under a reasonable assumption on the diffusion matrix $\mathbb{D}$. For a positive number $R$, we denote
\[
B_R(x_0,y_0)  = \{ (x,y) \in \Omega : |(x,y) - (x_0,y_0)| < R \} .
\]

\begin{lem}
\label{lem3}
Assume that $\Omega \subset \mathbb{R}^2 $ is a bounded domain and $\mathbb{D}$ is continuous
on $\overline{\Omega} \equiv \Omega \cup \partial \Omega$ and symmetric and uniformly positive definite
on $\Omega$. Then, there exists $R > 0$ such that for any $(x_0,y_0) \in \Omega$,
\begin{align}
& \sup_{B_R^+(x_0,y_0)} \frac{b(x,y)}{a(x,y)} + \frac{\overline{\alpha}}{3\alpha}
\le \inf_{B_R^+(x_0,y_0)} \frac{c(x,y)}{b(x,y)},
\label{lem3-1}\\
& \sup_{B_R^-(x_0,y_0)} \frac{c(x,y)}{b(x,y)} + \frac{\overline{\alpha}}{3\alpha}
\le \inf_{B_R^-(x_0,y_0)} \frac{b(x,y)}{a(x,y)} ,
\label{lem3-2}
\\
& \sup_{B_R^+(x_0,y_0)} \frac{b(x,y)}{c(x,y)} + \frac{\overline{\alpha}}{3\alpha}
\le \inf_{B_R^+(x_0,y_0)} \frac{a(x,y)}{b(x,y)},
\label{lem3-3}\\
& \sup_{B_R^-(x_0,y_0)} \frac{a(x,y)}{b(x,y)} + \frac{\overline{\alpha}}{3\alpha}
\le \inf_{B_R^-(x_0,y_0)} \frac{b(x,y)}{c(x,y)} ,
\label{lem3-4}
\end{align}
where
\begin{align}
\overline{\alpha} & = \inf_{\Omega} \left ( a(x, y) c(x,y) - b(x,y)^2\right ),
\label{alpha-1}
\\
\alpha & =  \max\left (\sup_{\Omega} a(x,y) (|b(x,y)|+1),\; \sup_{\Omega} c(x,y) (|b(x,y)|+1)\right ) .
\label{alpha-2}
\end{align}
\end{lem}
\begin{proof}
The inequalities (\ref{lem3-3}) and (\ref{lem3-4}) can be obtained from (\ref{lem3-1}) and (\ref{lem3-2})
by the symmetry between $a$ and $c$.
To prove (\ref{lem3-1}) and (\ref{lem3-2}), we first notice that the uniform positive definiteness of $\mathbb{D}$
implies that $\overline{\alpha} > 0$ and 
\[
a(x, y) c(x,y) - b(x,y)^2 \ge \overline{\alpha}, \quad \forall (x, y) \in \Omega .
\] 
Dividing both sides by $a b$, we have
\begin{align*}
& \frac{c(x,y)}{b(x, y)} \ge \frac{b(x,y)}{a(x, y)} + \frac{ \overline{\alpha} }{a(x,y) b(x,y)}, \quad \text{ for } b(x,y) > 0\\
& \frac{c(x,y)}{b(x, y)} \le \frac{b(x,y)}{a(x, y)} + \frac{ \overline{\alpha} }{a(x,y) b(x,y)}, \quad \text{ for }b(x,y) < 0 .
\end{align*}
Denoting 
\begin{equation}
F(x, y) = \frac{c(x,y)}{b(x, y)}, \quad G = \frac{b(x,y)}{a(x, y)},
\label{FG-1}
\end{equation}
from the definition of $\alpha$ (\ref{alpha-1}) we have
\begin{align*}
& F(x,y) \ge G(x, y) + \frac{\overline{\alpha}}{\alpha}, \quad \text{ for } b(x,y) > 0\\
& F(x,y) \le G(x,y) - \frac{\overline{\alpha}}{\alpha}, \quad \text{ for }b(x,y) < 0 .
\end{align*}
Define
\begin{equation}
M   = \sup_{\Omega} | G(x, y) | + \frac{\overline{\alpha}}{\alpha}.
\label{GM-1}
\end{equation}

To prove (\ref{lem3-1}), we introduce a cut-off function of $F$ as
\begin{equation}
F^+ (x,y) = 
	\begin{cases}
		F(x,y),	& \quad \text{ for } \quad b(x,y) > 0 \text{ and }  F(x,y) < M \\
		M, 	        & \quad \text{ for } \quad b(x,y) > 0 \text{ and }  F(x,y) \ge M \\
		M,            & \quad \text{ for } \quad b(x,y) \le 0 .
	\end{cases}
\label{F-1}
\end{equation}
This function is continuous on $\overline{\Omega}$ and has the following properties,
\begin{align}
& F^+(x, y) \le F(x,y), \quad \text{ for } b(x,y) > 0
\label{eq:F^+-1}\\
& F^+(x, y) \ge G(x, y) + \frac{\overline{\alpha}}{\alpha}, \quad \text{ for } (x,y) \in \Omega .
\label{eq:F^+-2}
\end{align}
Since $\Omega$ is bounded, $F^+$ and $G$ are uniformly continuous on $\overline{\Omega}$. 
Thus, there exists a constant $R_1 > 0$ (which depends on $F^+$, $G$, $\overline{\alpha}$,
and $\alpha$ and therefore the diffusion matrix $\mathbb{D}$ but not on $(x_0,y_0)$) such that
\begin{equation}
F^+(x,y) \ge F^+(x_0, y_0) - \frac{\overline{\alpha}}{3 \alpha},
\quad G(x, y) \le G(x_0, y_0) + \frac{\overline{\alpha}}{3\alpha},
\qquad \forall (x,y) \in B_{R_1}(x_0, y_0)
\label{FG-2}
\end{equation}
which in turn implies that
\[
\inf_{B_{R_1}(x_0,y_0) } F^+(x, y) \ge F^+(x_0, y_0) - \frac{\overline{\alpha}}{3 \alpha},\quad
\sup_{B_{R_1}(x_0,y_0) } G(x,y) \leq G(x_0, y_0) + \frac{\overline{\alpha}}{3 \alpha} .
\]
Combining these results with (\ref{eq:F^+-2}), we have
\[
\sup_{B_{R_1}(x_0,y_0) } G(x, y) \le G(x_0, y_0) + \frac{\overline{\alpha}}{3 \alpha}
\le F^+(x_0, y_0) -\frac{2 \overline{\alpha}}{3 \alpha}
\le \inf_{B_{R_1}(x_0, y_0) } F^+(x,y) - \frac{\overline{\alpha}}{3 \alpha}
\]
or
\[ 
\sup_{B_{R_1}(x_0, y_0) } G(x, y) \le \inf_{B_{R_1}(x_0, y_0)} F^+(x,y) - \frac{\overline{\alpha}}{3 \alpha}.
\]
From this and (\ref{eq:F^+-1}), we have
\begin{align*}
\sup_{B_{R_1}^+(x_0, y_0) } G(x, y) & \le \sup_{B_{R_1}(x_0, y_0) } G(x, y) 
 \le \inf_{B_{R_1}(x_0, y_0)} F^+(x,y) - \frac{\overline{\alpha}}{3 \alpha} 
\\
& \le \inf_{B_{R_1}^+(x_0, y_0)} F^+(x,y) - \frac{\overline{\alpha}}{3 \alpha} 
 \le \inf_{B_{R_1}^+(x_0, y_0)} F(x,y) - \frac{\overline{\alpha}}{3 \alpha},
\end{align*}
which gives (\ref{lem3-1}) (with $R = R_1$).

Similarly, we define
\begin{equation}
F^- (x,y) = 
	\begin{cases}
		F(x,y),	& \quad \text{ for } \quad b(x,y) < 0 \text{ and }  F(x,y) > -M \\
		-M, 	        & \quad \text{ for } \quad b(x,y) < 0 \text{ and }  F(x,y) \le  -M \\
		-M,            & \quad \text{ for } \quad b(x,y) \ge 0 .
	\end{cases}
\label{F-2}
\end{equation}
It is continuous on $\overline{\Omega}$ and satisfies
\begin{align}
& F^-(x, y) \ge F(x,y), \quad \text{ for } b(x,y) < 0
\label{eq:F^--1}\\
& F^-(x, y) \le G(x, y) - \frac{\overline{\alpha}}{\alpha}, \quad \text{ for } (x,y) \in \Omega .
\label{eq:F^--2}
\end{align}
From the uniform continuity of $F^-$ and $G$ on $\overline{\Omega}$,
there exists a constant $R_2 > 0$ (which depends only on the diffusion matrix $\mathbb{D}$)
such that
\begin{equation}
F^-(x,y) \le F^-(x_0, y_0) + \frac{\overline{\alpha}}{3 \alpha},
\quad G(x, y) \ge G(x_0, y_0) - \frac{\overline{\alpha}}{3 \alpha},\quad \forall (x, y)
\in B_{R_2}(x_0, y_0) .
\label{FG-3}
\end{equation}
Thus,
\[
\sup_{B_{R_2}(x_0, y_0)} F^-(x, y) \le F^-(x_0, y_0) + \frac{\overline{\alpha}}{3 \alpha}
\le G(x_0, y_0) - \frac{2\overline{\alpha}}{3 \alpha} \le \inf_{B_{R_2}(x_0, y_0)} G(x, y) - \frac{\overline{\alpha}}{3 \alpha}
\]
or
\[
\sup_{B_{R_2}(x_0, y_0)} F^-(x, y) \le \inf_{B_{R_2}(x_0, y_0)} G(x, y) - \frac{\overline{\alpha}}{3 \alpha}.
\]
Then, the inequality (\ref{lem3-2}) (with $R = R_2$) follows from this and (\ref{eq:F^--1}).

Using the symmetry between $a$ and $c$ we can show from (\ref{lem3-1}) and (\ref{lem3-2})
that (\ref{lem3-3}) and (\ref{lem3-4}) hold with some positive numbers $R_3$ and $R_4$, respectively.
Taking $R = \min( R_1, R_2, R_3, R_4)$, we have proven that (\ref{lem3-1}) -- (\ref{lem3-4}) hold
for the same $R$.
\end{proof}

\vspace{5pt}

\begin{rem}
\label{rem2.2}
The constants $\overline{\alpha}$ and $\alpha$ defined in (\ref{alpha-1}) and (\ref{alpha-2}) and
the functions $F$, $F^+$, $F^{-}$, and $G$ defined in (\ref{FG-1}), (\ref{GM-1}), (\ref{F-1}), and (\ref{F-2})
are determined completely by the diffusion matrix $\mathbb{D}$. So is the constant $R$.
If $F^+$, $F^-$, and $G$ are Lipschitz continuous with constants $L_{F^+}$, $L_{F^-}$,
and $L_{G}$, respectively, then we can choose $R$ as
\begin{equation}
R = \frac{\overline{\alpha}}{3 \alpha \max (L_{F^+}, L_{F^-}, L_{G})} .
\label{R-1}
\end{equation}
\qed
\end{rem}

\begin{thm}
\label{thm1}
Assume that $\Omega \subset \mathbb{R}^2 $ is a bounded domain and the diffusion matrix $\mathbb{D}$
is continuous on $\overline{\Omega}$ and symmetric and uniformly positive definite on $\Omega$.
Then, there exists a constant $R > 0$ (which depends only on $\mathbb{D}$) such that for any
$(x_0, y_0) \in \Omega$, the intervals
\[
\left ( \ds \sup_{B_R^+(x_0,y_0)} \dfrac{b(x,y)}{a(x,y)}, \quad \inf_{B_R^+(x_0,y_0)} \dfrac{c(x,y)}{b(x,y)}\right ),
\qquad \left (\ds \sup_{B_R^-(x_0,y_0)} \dfrac{c(x,y)}{b(x,y)}, \quad \inf_{B_R^-(x_0,y_0)} \dfrac{b(x,y)}{a(x,y)}\right )
\]
are not empty when $B_R^+(x_0,y_0)$ and $B_R^-(x_0,y_0)$ are not empty, respectively. Moreover, 
for any constants $\beta_1$ and $\beta_2$ with $\tan \beta_1$ and $\tan \beta_2$ in the intervals,
the nonnegative directional splitting (\ref{splitting-1}) holds with $\Omega_0 = B_{R}(x_0, y_0)$
and nonnegative coefficients given in (\ref{gamma0}) -- (\ref{gamma2}).
\end{thm}

\begin{proof}
The theorem is obtained by combining Lemma~\ref{lem2} and Lemma~\ref{lem3}.
\end{proof}

\vspace{5pt}

\begin{rem}
\label{rem2.3}
The result in the above theorem is different from Weickert's result \cite{Weickert}
in the sense that the coefficients in the directional splitting (\ref{splitting-1}) are nonnegative
in a neighborhood of $(x_0,y_0)$ whereas those of (\ref{eqn:split0}) are nonnegative only at
$(x_0,y_0)$. This is crucial to the construction of monotone FD schemes
since any FD scheme for the divergence form (\ref{had}) uses the values of the coefficients on a neighborhood or a stencil
of a mesh point for the discretization at the point.
\qed
\end{rem}

\begin{rem}
\label{rem2.4}
As mentioned in Remark~\ref{rem2.2} the constant $R$ depends on the diffusion matrix $\mathbb{D}$
and is independent of the location. This implies that a uniform mesh can be used to construct a monotone
FD scheme as long as the mesh spacing is sufficiently small and the maximal size of the FD stencils (used for all mesh points)
is bounded (see \S\ref{sec:FD}).
\qed
\end{rem}

\section{Construction of a monotone FD discretization}
\label{sec:FD}

In this section we develop a monotone FD discretization for (\ref{had}) based on the nonnegative directional splitting
(\ref{splitting-1}). We consider a uniform mesh $(x_j,y_k) = (j h, k h),\;  j, k = 0, ..., N$,
where $N$ is a positive integer and $h = 1/N$. 

Let $(x_j, y_k)$ be an interior mesh point. Denote by $\omega_{j,k}$ a FD stencil
of size $(2 m_{j,k}+1) \times (2 m_{j,k}+1)$ centered at $(x_j, y_k)$.
We assume that 
\begin{equation}
m_{j,k} \le \left [ \frac{3 \alpha}{\overline{\alpha}}\right ] + 1,
\label{m-1}
\end{equation}
where $\overline{\alpha}$ and $\alpha$ are defined in (\ref{alpha-1}) and (\ref{alpha-2})
and $[x]$ denotes the nearest integer smaller than or equal to $x$. We will show that 
the right-hand side of (\ref{m-1}) is an upper bound of the stencil size
to guarantee a monotone FD discretization for (\ref{had}).
We also assume that the mesh spacing $h$ is sufficiently small such that
the stencil $\omega_{j,k} \subset B_R(x_j,y_k)$,
with $R$ being the radius determined in Lemma~\ref{lem3}
and Theorem~\ref{thm1}. Mathematically, this is equivalent to
\begin{equation}
\sqrt{2} h \max\limits_{j, k} m_{j,k} \le R.
\label{h-1}
\end{equation}
Using (\ref{R-1}) and (\ref{m-1}), we obtain a sufficient condition for (\ref{h-1}) as
\begin{equation}
\sqrt{2} h \left ( \left [ \frac{3 \alpha}{\overline{\alpha}}\right ] + 1\right ) \le 
\frac{\overline{\alpha}}{3 \alpha \max (L_{F^+}, L_{F^-}, L_{G})} ,
\label{h-2}
\end{equation}
which can be achieved by choosing $h$ sufficiently small.

We first consider the case where $\omega_{j,k} \subset \overline{\Omega}$.
Under the above assumptions, Theorem~\ref{thm1} implies that the nonnegative directional splitting
(\ref{splitting-1}), with the coefficients being nonnegative on $\omega_{j,k}$, holds
for any constants $\beta_1$ and $\beta_2$ satisfying 
\begin{align} 
A_{j,k} \equiv \ds \sup_{\omega_{j,k}^+} \dfrac{b(x,y)}{a(x,y)}
< \tan\beta_1 < B_{j,k} \equiv \inf_{\omega_{j,k}^+} \dfrac{c(x,y)}{b(x,y)},
\label{AB-1}
\\
C_{j,k}\equiv \ds \sup_{\omega_{j,k}^-} \dfrac{c(x,y)}{b(x,y)}
< \tan\beta_2< D_{j,k} \equiv \inf_{\omega_{j,k}^-} \dfrac{b(x,y)}{a(x,y)} .
\label{CD-1}
\end{align}
It is not difficult to see that
\[
A_{j,k}^{-1} = \inf_{\omega_{j,k}^+} \dfrac{a(x,y)}{b(x,y)},
\quad B_{j,k}^{-1} = \sup_{\omega_{j,k}^+} \dfrac{b(x,y)}{c(x,y)},
\quad
C_{j,k}^{-1} = \inf_{\omega_{j,k}^-} \dfrac{b(x,y)}{c(x,y)},
\quad D_{j,k}^{-1} = \sup_{\omega_{j,k}^-} \dfrac{a(x,y)}{b(x,y)} .
\]
From Lemma~\ref{lem3}, we have
\begin{align}
& B_{j,k} - A_{j,k} \ge \frac{\overline{\alpha}}{3 \alpha}, \quad 
A_{j,k}^{-1} - B_{j,k}^{-1} \ge \frac{\overline{\alpha}}{3 \alpha},
\label{AB-2}
\\
& D_{j,k} - C_{j,k} \ge \frac{\overline{\alpha}}{3 \alpha}, \quad 
C_{j,k}^{-1} - D_{j,k}^{-1} \ge \frac{\overline{\alpha}}{3 \alpha} ,
\label{CD-2}
\end{align}
where $\overline{\alpha}$ and $\alpha$ are defined in (\ref{alpha-1}) and (\ref{alpha-2}). 

The central FD discretization of the first and fourth terms of (\ref{splitting-1}) is straightforward. It can be done
based on points $(x_{j-1}, y_k), \; (x_j, y_k), \; (x_{j+1}, y_k)$ and $(x_{j}, y_{k-1}), \; (x_j, y_k), \; (x_j, y_{k+1})$,
respectively. (These points are on a $3\times 3$ stencil centered at $(x_j,y_k)$.)
On the other hand, the FD discretization of the second and third terms are more complicated.
Take the second term as an example. The main challenge of the discretization is to find
$\beta_1$ with $\tan \beta_1 \in (A_{j,k}, B_{j,k})$ such that the straight line having the slope $\tan \beta_1$
and passing the point $(x_j,y_k)$ intersects at least two more mesh points on $\omega_{j,k}$ other than $(x_j,y_k$).
Once such a $\beta_1$ has been found, the three mesh points are then used to 
discretize the term (which is the second order directional derivative along $\tan \beta_1$) using central finite differences.

\begin{figure}[thb] 
\centering
\includegraphics[scale=0.4]{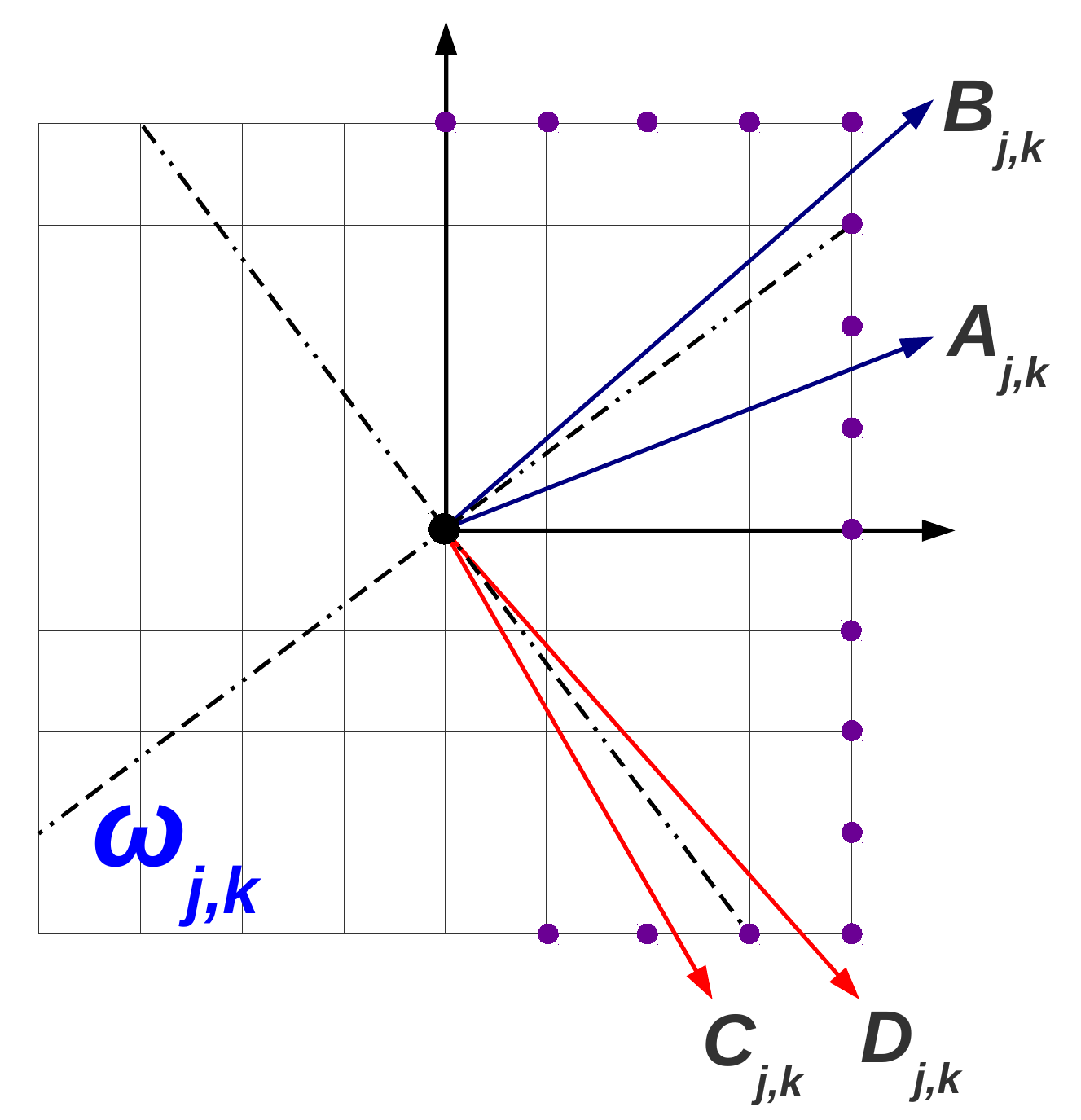}
\caption{Sketch of a FD stencil and principal directions.
\label{fig:stencil}}
\end{figure}
	
To find $\beta_1$, we follow \cite{Weickert} to define $4 m_{j,k}$ principal
directions associated with the outer mesh nodes of $\omega_{j,k}$ (cf. Fig.~\ref{fig:stencil}),
\begin{equation}
{\label{betas}}
	\hat{\beta}_i = \begin{cases}
\arctan\left( \frac{i}{m_{j,k}} \right)\,, & \quad \text{for} \quad -m_{j,k} \leq i \leq m_{j,k} \\
\arctan\left( \frac{m_{j,k}}{2m_{j,k}-i} \right)\,, & \quad \text{for} \quad m_{j,k}+1 \leq i \leq 2m_{j,k} \\
\arctan\left( \frac{m_{j,k}}{-2m_{j,k}-i} \right)\,, & \quad \text{for} \quad -2m_{j,k} + 1 \leq i \leq -m_{j,k}-1 .
			   \end{cases}
\end{equation}
When $m_{j,k}$ is sufficiently large, we can choose $\beta_1$ to be one of $\hat{\beta}_i$, $i=1, ..., 2 m_{j,k}$
that satisfies (\ref{AB-1}).
Indeed, there are three possibilities $A_{j,k}$ and $B_{j,k}$ are located:
$A_{j,k} < 1 < B_{j,k}$, $ A_{j,k} < B_{j,k} \le 1$, or $1\le A_{j,k} < B_{j,k}$.
For the case $A_{j,k} < 1 < B_{j,k}$, we can choose $\beta_1 = \hat{\beta}_{m_{j,k}} = \pi/4$ and $m_{j,k} = 1$.
It is easy to see that (\ref{AB-1}) is satisfied and the three mesh points on the diagonal line running from
the southwest corner to the northeast corner can be used to discretize the directional derivative.

For the case $ A_{j,k} < B_{j,k} \le 1$, we need to consider the principal directions $\hat{\beta}_i$, $i=1, ..., m_{j,k}-1$.
Thus, we require that at least there is an $i$ $(1 \le i < m_{j,k}-1)$ such that
\[
  A_{j,k} < \frac{i}{m_{j,k}} < B_{j,k} .
\]
(In this case, $\beta_1$ is chosen as $\hat{\beta}_i$.)
Equivalently, $m_{j,k}$ should be sufficiently large such that $(m_{j,k} A_{j,k}, m_{j,k} B_{j,k})$ contains at least an integer.
A sufficient condition for this is that the length of the interval is bigger than one, i.e.,
\begin{equation}
m_{j,k} > \frac{1}{B_{j,k} - A_{j,k}} .
\label{eqn:mAB-2}
\end{equation}

For the case $1\le A_{j,k} < B_{j,k}$, we need to consider the principal directions $\hat{\beta}_i$, $i=m_{j,k}+1, ..., 2 m_{j,k}-1$.
We require
\[
A_{j,k} < \frac{m_{j,k}}{2 m_{j,k}-i} < B_{j,k}
\]
or 
\[
m_{j,k} B_{j,k}^{-1} < 2 m_{j,k}-i < m_{j,k} A_{j,k}^{-1}
\]
for some $m_{j,k} < i < 2 m_{j,k}$. This is mathematically equivalent to the requirement that
$(m_{j,k} B_{j,k}^{-1}, m_{j,k} A_{j,k}^{-1})$ contains at least an integer.
A sufficient condition for this is
\[
m_{j,k} > \frac{1}{A_{j,k}^{-1}-B_{j,k}^{-1}} .
\]

The above analysis shows that we should choose $m_{j,k}$ such that
\begin{equation}
\begin{cases}
m_{j,k} = 1, & \quad \text{ for } A_{j,k} < 1 < B_{j,k}
\\
(m_{j,k} A_{j,k}, m_{j,k} B_{j,k}) \text{ contains at least an integer}, & \quad \text{ for } A_{j,k} < B_{j,k} \le 1
\\
(m_{j,k} B_{j,k}^{-1}, m_{j,k} A_{j,k}^{-1}) \text{ contains at least an integer}, & \quad \text{ for } 1 \le A_{j,k} < B_{j,k} .
\end{cases}
\label{m-2}
\end{equation}
A sufficient condition for this is
\begin{equation}
\begin{cases}
m_{j,k} \ge 1, & \quad \text{ for } A_{j,k} < 1 < B_{j,k}
\\
m_{j,k} > \frac{1}{B_{j,k} - A_{j,k}}, & \quad \text{ for } A_{j,k} < B_{j,k} \le 1
\\
m_{j,k} > \frac{1}{A_{j,k}^{-1} - B_{j,k}^{-1}}, & \quad \text{ for } 1 \le A_{j,k} < B_{j,k} .
\end{cases}
\label{m-3}
\end{equation}
From (\ref{AB-2}), one can see that the above condition can be satisfied when $m_{j,k}$ is taken as
$m_{j,k} = \left [ \frac{3 \alpha}{\overline{\alpha}}\right ] + 1$. Since this is only a sufficient condition,
so we have (\ref{m-1}).

Similarly, the condition that we can choose $\beta_2$ as one of the principal directions $\hat{\beta}_i$, $i = -2 m_{j,k}, ..., -1$
satisfying (\ref{CD-1}) is
\begin{equation}
\begin{cases}
m_{j,k} = 1, & \quad \text{ for } C_{j,k} < -1 < D_{j,k}
\\
(m_{j,k} C_{j,k}, m_{j,k} D_{j,k}) \text{ contains at least an integer}, & \quad \text{ for } -1 \le C_{j,k} < D_{j,k}
\\
(m_{j,k} D_{j,k}^{-1}, m_{j,k} C_{j,k}^{-1}) \text{ contains at least an integer}, & \quad \text{ for } C_{j,k} < D_{j,k} \le -1 .
\end{cases}
\label{m-4}
\end{equation}
A sufficient condition for this is
\begin{equation}
\begin{cases}
m_{j,k} \ge 1, & \quad \text{ for } C_{j,k} < -1 < D_{j,k}
\\
m_{j,k} > \frac{1}{D_{j,k} - C_{j,k}}, & \quad \text{ for } -1 \le C_{j,k} < D_{j,k}
\\
m_{j,k} > \frac{1}{C_{j,k}^{-1} - D_{j,k}^{-1}}, & \quad \text{ for }  C_{j,k} < D_{j,k} \le -1 .
\end{cases}
\label{m-5}
\end{equation}
From (\ref{CD-2}), one can see that the above condition can be satisfied when $m_{j,k}$ is taken as
the right-hand side of (\ref{m-1}).

\begin{figure}[thb]
\begin{center}
\hspace{8mm}
\includegraphics[scale=0.4]{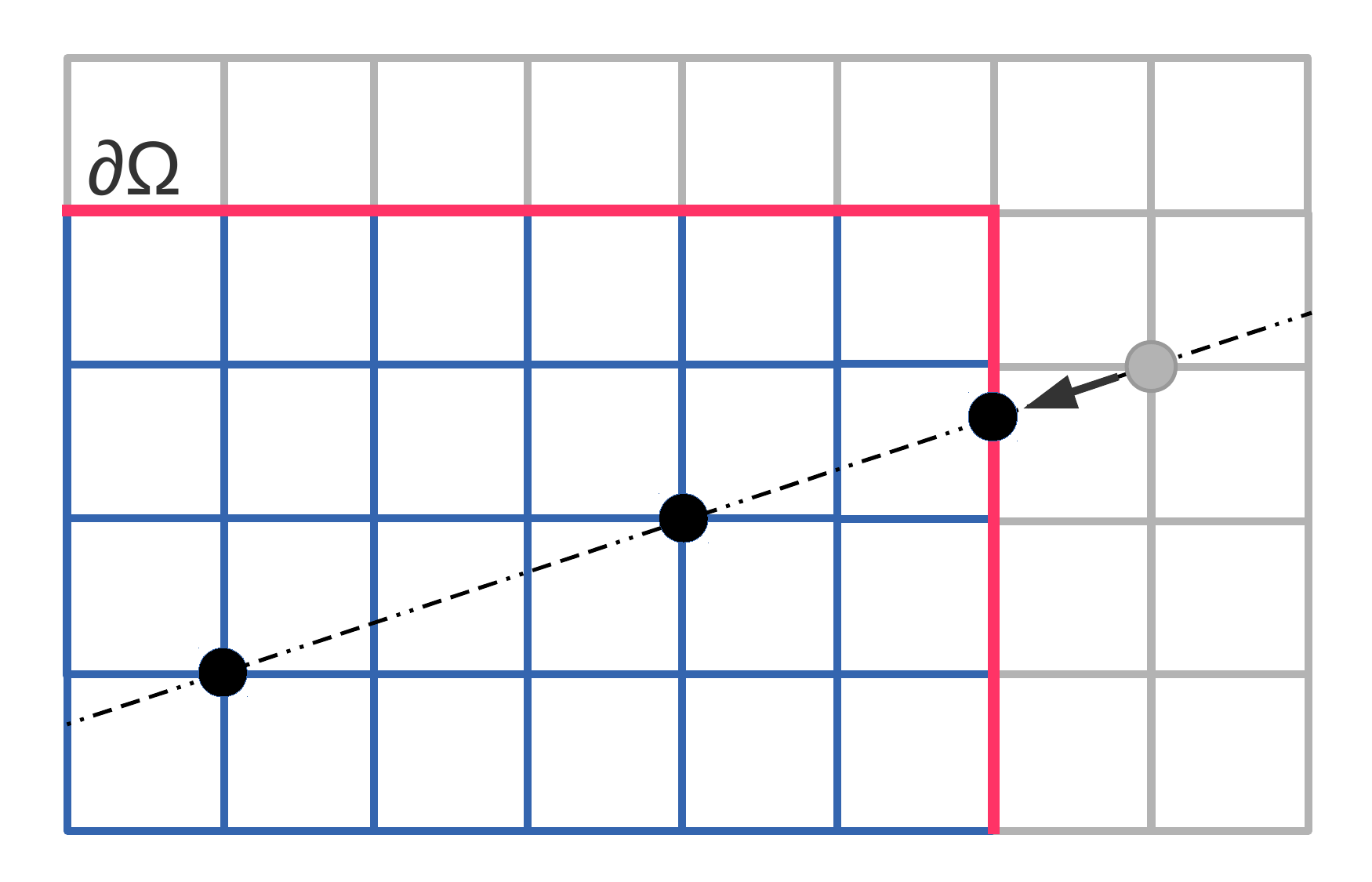}
\caption{FD stencil near the domain boundary}
\label{fig:bdry}
\end{center}
\end{figure}

We now consider the situation when the stencil $\omega_{j,k}$ is not in $\overline{\Omega}$. 
This can happen when $(x_j, y_k)$ is close to the domain boundary and $m_{j,k} > 1$. 
In this case, we can determine $m_{j,k}$, $\beta_1$, and $\beta_2$ as for the case where
$\omega_{j,k}$ is in $\overline{\Omega}$. But, one of the three mesh points found along direction
$\tan \beta_1$ or $\tan \beta_2$ lays outside of the domain; see Fig.~\ref{fig:bdry}.
A treatment for this difficulty is to replace the mesh point by the intersection of the straight line
(with slope $\tan \beta_1$ or $\tan \beta_2$) and the domain boundary and then discretize
the corresponding directional derivative using central finite differences based on this boundary
point and other two mesh points.

It is worth pointing out that each of the terms on the right-hand side of (\ref{splitting-1}) has been discretized
using central three-point finite differences along the corresponding direction. Since the coefficients of the
differential operator are nonnegative, it is not difficult to see that the coefficient matrix of the FD equations is
a $Z$-matrix (with off-diagonal entries being nonpositive) and has nonnegative diagonal entries.
Moreover, it is easy to verify that the matrix is diagonally dominant and irreducible. Thus, the coefficient
matrix is an $M$-matrix and the FD scheme is monotone.

For easy reference, we summarize the above analysis in the following theorem.

\begin{thm}
\label{thm3.1}
Assume that the diffusion matrix $\mathbb{D}$ is continuous on $\overline{\Omega}$
and symmetric and uniformly positive definite on $\Omega$ and a uniform mesh with spacing $h$ is given for $\Omega$.
If $h$ is sufficiently small, a monotone FD scheme exists for problem (\ref{had}) and can be constructed based on
central (directional) finite differences.
\end{thm}

\begin{rem}
\label{rem3.1}
Generally speaking, the upper bound for $m_{j,k}$ defined in (\ref{m-1}) is very conservative.
In practice, we can use stencils of variable size for different mesh points. The stencil size
$m_{j,k}$ can be found using (\ref{m-2}) and (\ref{m-4}). The constant $R$ can be calculated
using (\ref{R-1}). The mesh spacing can be verified using
\begin{equation}
\sqrt{2} h \max\limits_{j,k} m_{j,k} \le R .
\label{h-3}
\end{equation}
\qed
\end{rem}

To conclude this section, we would like to see what condition is
if only the three-by-three ($m_{j,k}=1$) stencil is used for all interior mesh points.
In this case, we need to use the mesh points on the diagonal lines of $\omega_{j,k}$
to discretize the second and third terms on the right-hand side of (\ref{splitting-1}).
This means that $A_{j,k} < 1 < B_{j,k}$ ($\tan \beta_1 = 1$) and $C_{j,k} < -1 < D_{j,k}$
($\tan \beta_2 = -1$). From (\ref{AB-1}) and (\ref{CD-1}), this implies
\begin{equation}
a(x,y) > |b(x,y)|, \quad c(x,y) > |b(x,y)|, \quad \forall (x,y) \in \Omega
\label{cond-1}
\end{equation}
or, in words, $\mathbb{D}$ is strictly diagonally dominant for any point in $\Omega$.
Thus, we have the following theorem.

\begin{thm}
\label{thm3.2}
Assume that the diffusion matrix $\mathbb{D}$ is continuous on $\overline{\Omega}$
and symmetric and uniformly positive definite on $\Omega$ and a uniform mesh with spacing $h$ is given thereon.
If (\ref{cond-1}) is satisfied and $h$ is sufficiently small, then a monotone FD scheme can be constructed
for (\ref{had}) based on a three-by-three stencil and central finite differences.
\end{thm}

It is pointed out that the condition (\ref{cond-1}) has been obtained by Weickert \cite{Weickert} and
Mr\'{a}zek and Navara \cite{Mrazek}.


\section{Numerical results}
\label{sec:num}

In this section we present some numerical results for four examples to verify the theoretical findings in the previous
sections. We first note that it is useful to have an estimate on the values of parameters $R$ (the uniform radius
defined in Lemma~\ref{lem3}) and $m_{j,k}$ (the size of FD stencil) in the actual computation.
The former can be estimated using (\ref{R-1}) where $\bar{\alpha}$, $\alpha$, $L_{F^+}$, $L_{F^-}$, and $L_{G}$
can be approximated numerically. For the latter, (\ref{m-1}) is generally too conservative to use in practical
computation. Instead, we use (\ref{m-2}) and (\ref{m-4}) to estimate $m_{j,k}$, with $A_{j,k}$, $B_{j,k}$,
$C_{j,k}$, $D_{j,k}$ being computed on $B_{R}(x_j,y_k)$. Notice that we use  $B_{R}(x_j,y_k)$ instead of $\omega_{j,k}$
as defined in (\ref{AB-1}) and (\ref{CD-1}). This is because the former is larger than the latter and the former
is independent of $m_{j,k}$ which is to be sought.
Once $m_{j,k}$ has been determined, we can find two principal directions (cf. (\ref{betas})) to satisfy (\ref{AB-1}) and (\ref{CD-1})
and finally construct the FD discretization of (\ref{had}).

\begin{exam}
\label{exam1}
This example is in the form of (\ref{had}) with 
\begin{equation}
\mathbb{D} = \begin{bmatrix}  
			9 & 4 \sin(2 \pi x y) \\ 
			4 \sin(2 \pi x y) & 3 
\end{bmatrix} .
\label{exam1-D}
\end{equation}
We seek to verify that the developed monotone scheme satisfy DMP and choose
\[ 
  f(x,y) = 0 \,, \quad g(x,y) = \cos(\pi x y) + y .
\]
This example does not satisfy (\ref{cond-1}). So we do not expect that a $3 \times 3 $ stencil will work
for this example. A direct numerical calculation shows that
\[
\overline{\alpha}  = 11,\qquad \alpha = 45.
\]
Thus the upper bound in (\ref{m-1}) is $[2 \alpha/\overline{\alpha}] + 1 = 13$, which requires
a stencil of size $27 \times 27$ for the monotone FD discretization. On the other hand,
a numerical calculation based on (\ref{m-2}) and (\ref{m-4}) shows that a stencil of size $5 \times 5$
is sufficient to guarantee a monotone FD discretization. 

The numerical results obtained with a stencil of size $5 \times 5$
for this example are shown in Table~\ref{exam1-table}. They confirm
that the computed solution satisfies the maximum principle and stays between the minimum and maximum values
of the solution on the domain boundary.

\begin{table}[thb]
\begin{center}

\begin{tabular}{|c|c|c|c|c|}
\hline
    $J=K$	   & Boundary Min  & Interior Min  & Boundary Max & Interior Max \\
\hline 21     &   -5.105652e-2 &  1.040961e-2  &   2.000000 &  1.912261 \\ 
\hline 51     &   -5.105652e-2 & -2.444841e-2  &   2.000000 &  1.972163 \\ 
\hline 101   &   -5.105652e-2 & -3.753179e-2  &   2.000000 &  1.987642 \\ 
\hline 501   &   -5.109830e-2 & -4.837255e-2  &   2.000000 &  1.997862 \\ 
\hline 1001 &   -5.110190e-2 & -4.973654e-2  &   2.000000 &  1.998961 \\ 
\hline
\end{tabular}
\captionof{table}{Example \ref{exam1}. Extrema of the numerical solutions in the interior and on the boundary of the physical domain.}
\label{exam1-table}
\end{center}
\end{table}
\end{exam}

\begin{exam}
\label{exam2}
In this example, we seek to verify the accuracy of the monotone scheme.
We use the same differential operator as in Example \ref{exam1}, i.e., in the form of (\ref{had}) with 
(\ref{exam1-D}), but choose $f$ and $g$ such that the exact solution of the BVP is given by
$u(x,y) = \sin(2 \pi x) \sin(3 \pi y)$.

The numerical results obtained for this example are shown in  
Figs.~\ref{exam2-f1} and \ref{exam2-f2}. The contours of the exact and computed solutions
are shown in Fig.~\ref{exam2-f1}. The sign pattern in Fig.~\ref{exam2-f2a} shows that the coefficient
matrix has nonpositive off-diagonal entries and nonnegative diagonal entries, as stated in Theorem~\ref{thm3.1}.
Since the scheme is monotone, from the result of Barles and Souganidis \cite{Barles}
we may expect the computed solution to converge
in second order, which indeed can be seen from the convergence history shown in Fig.~\ref{exam2-f2b}.

\begin{figure}
        \centering
        \begin{subfigure}[b]{0.5\linewidth}
                \includegraphics[width=\textwidth]{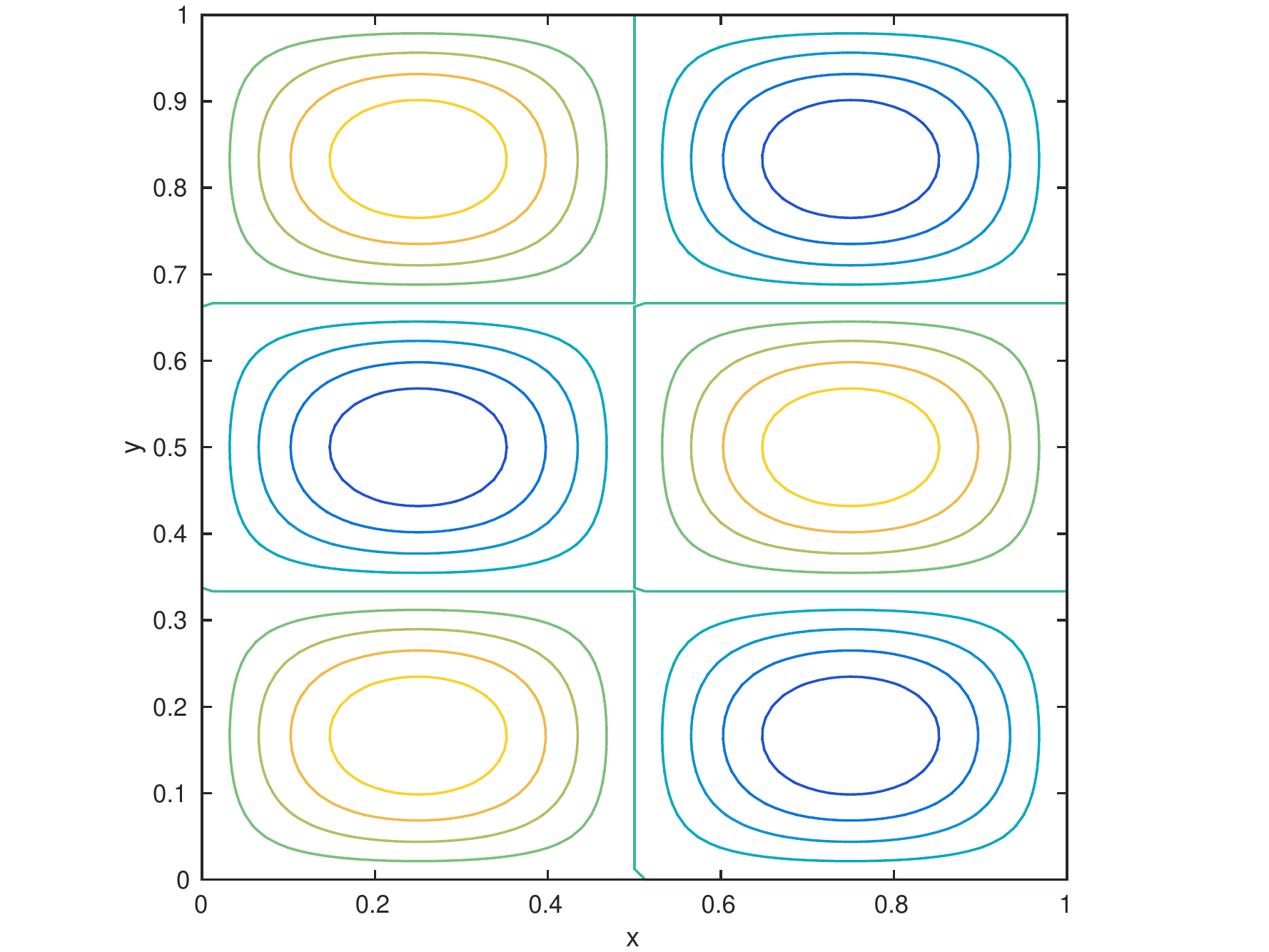}
                \caption{Exact solution}
                \label{exam2-f1a}
        \end{subfigure}%
        \begin{subfigure}[b]{0.5\linewidth}
                \includegraphics[width=\textwidth]{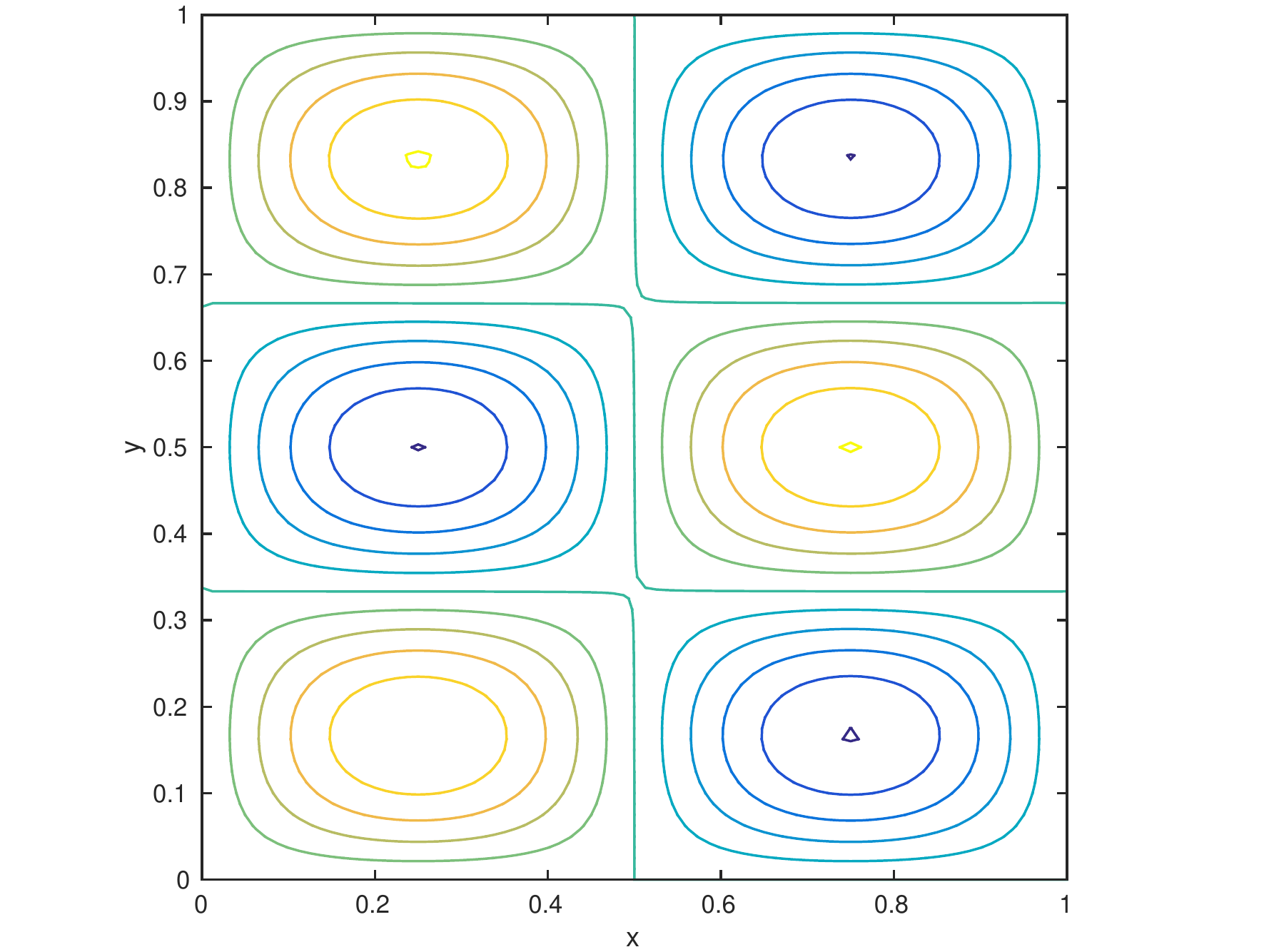}
                \caption{Numerical solution with $J = 81$}
                \label{exam2-f1b}
        \end{subfigure}
	\caption{Example \ref{exam2}. Contours of the exact and computed solutions.}
	\label{exam2-f1}
\end{figure}

\begin{figure}
        \centering
        \begin{subfigure}[b]{0.5\linewidth}
                \includegraphics[width=\textwidth]{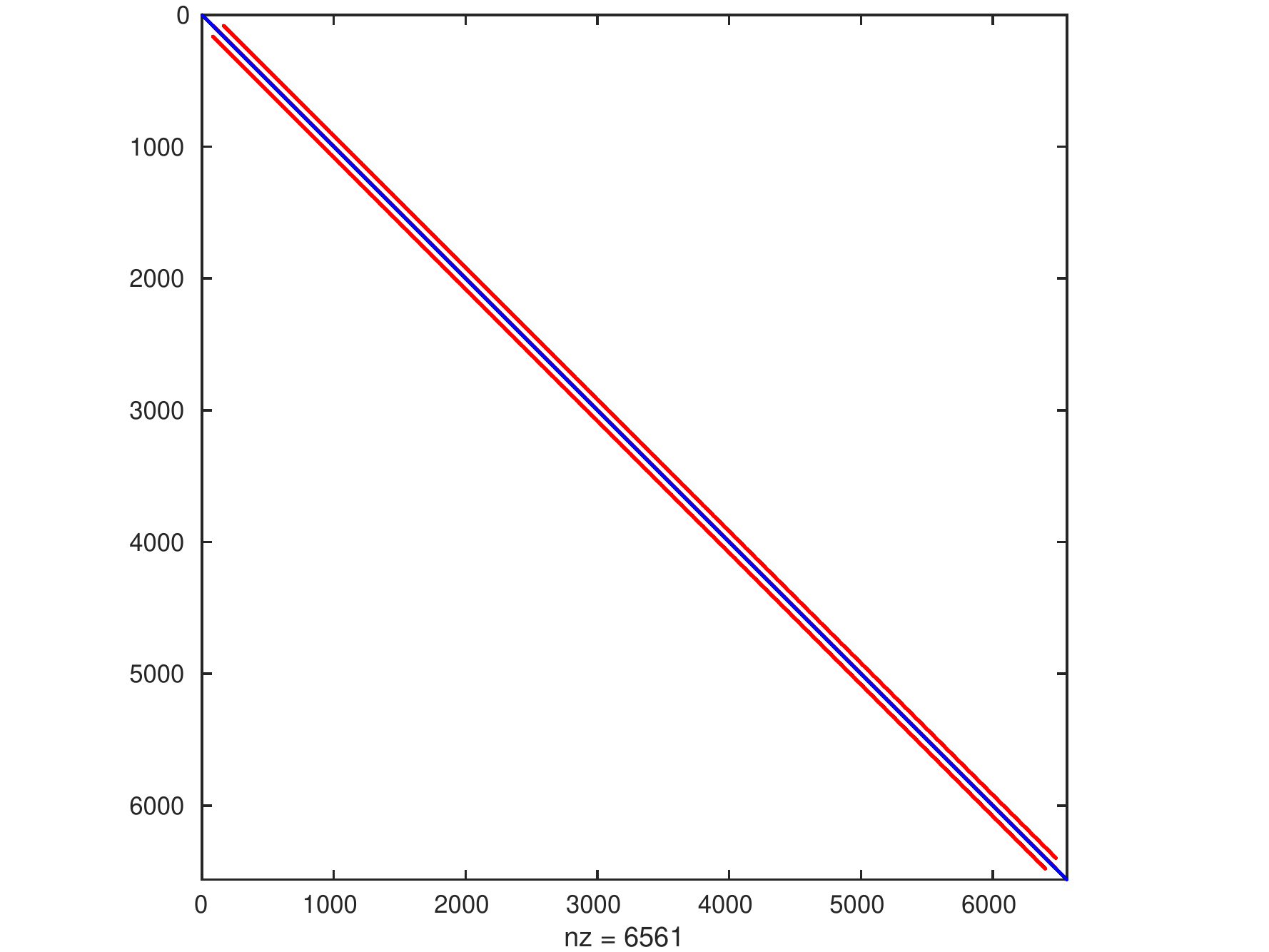}
                \caption{Sign pattern for J = 81}
                \label{exam2-f2a}
        \end{subfigure}%
        \begin{subfigure}[b]{0.5\linewidth}
                \includegraphics[width=\textwidth]{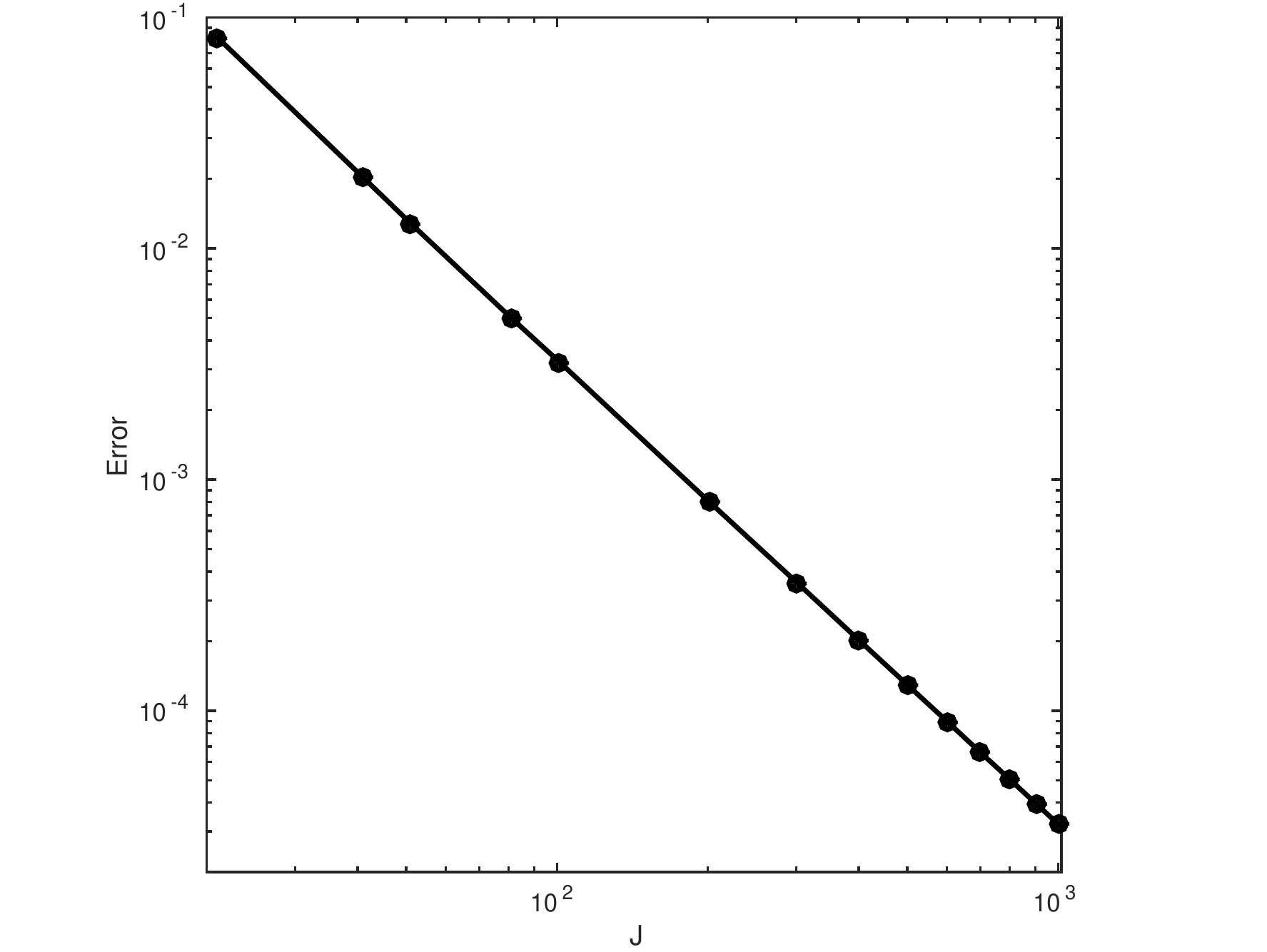}
                \caption{Convergence history}
                \label{exam2-f2b}
        \end{subfigure}
	\caption{Example \ref{exam2}. Sign pattern of the coefficient matrix and the convergence history
		 	of the maximum norm of the error  of the FD solution.}
	\label{exam2-f2}
\end{figure}

\end{exam}

\begin{exam}
\label{exam3}
For this example, the diffusion matrix is given by
\begin{equation}
	\mathbb{D} = \begin{bmatrix}
			1.1 & \sin(2 \pi xy) \\ 
			\sin(2 \pi xy) & 1.1
		\end{bmatrix} .
\end{equation}
The functions $f$ and $g$ are chosen such that the exact solution of the BVP reads as 
$ u(x,y) = \sin(2 \pi x) \sin(3 \pi y) $. This example satisfies (\ref{cond-1}) and by Theorem \ref{thm3.2},
we can use a $3 \times 3$ stencil for constructing a monotone scheme. 

The numerical results for this example are shown in Figs.~\ref{exam3-f1} and \ref{exam3-f2}.
One can see that the coefficient matrix has nonpositive off-diagonal entries and nonnegative diagonal entries
and the computed solution converges to the exact solution at a rate of second order.

\begin{figure}
        \centering
        \begin{subfigure}[b]{0.5\linewidth}
                \includegraphics[width=\textwidth]{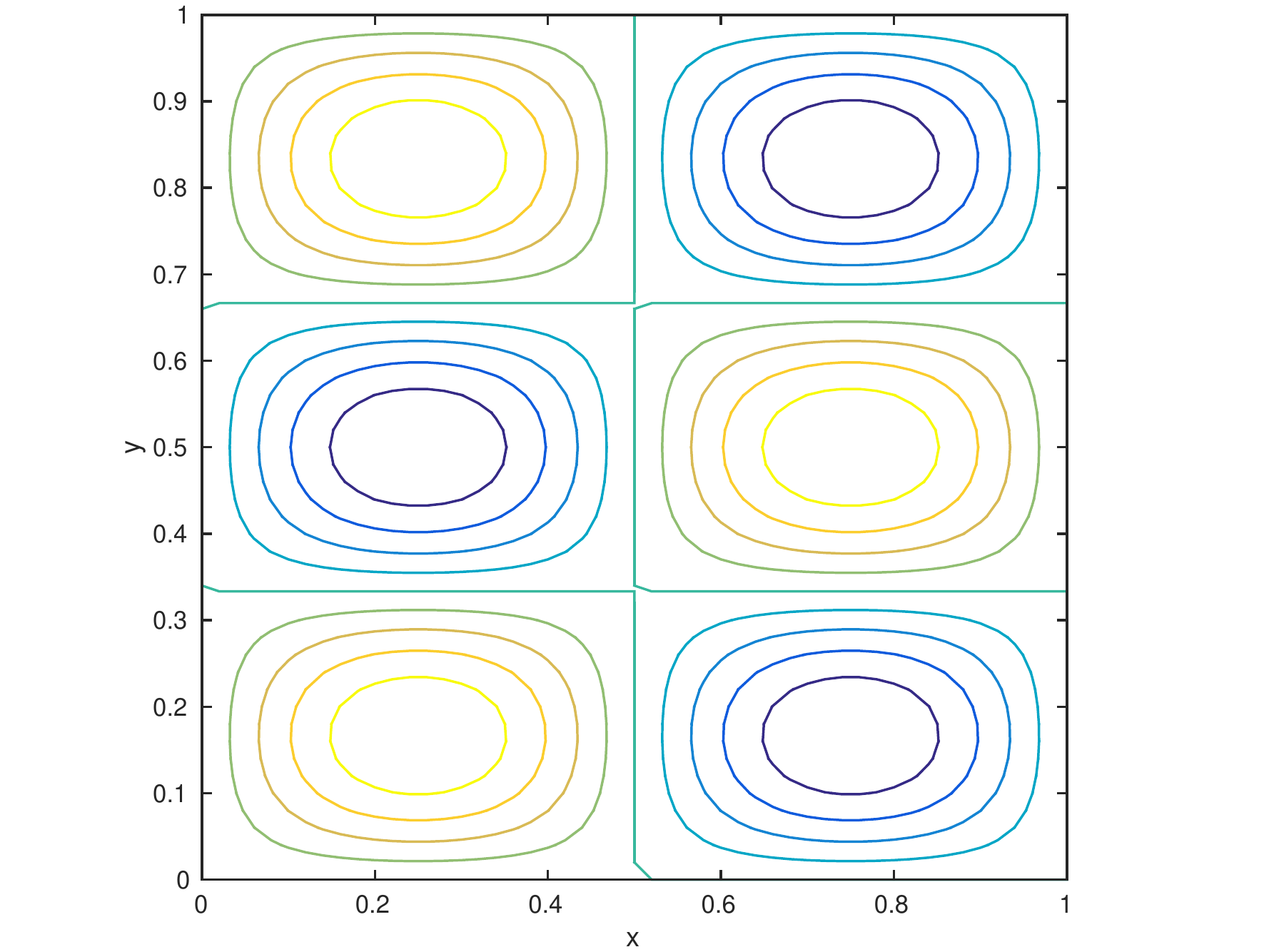}
                \caption{Exact solution}
                \label{exam3-f1a}
        \end{subfigure}%
        \begin{subfigure}[b]{0.5\linewidth}
                \includegraphics[width=\textwidth]{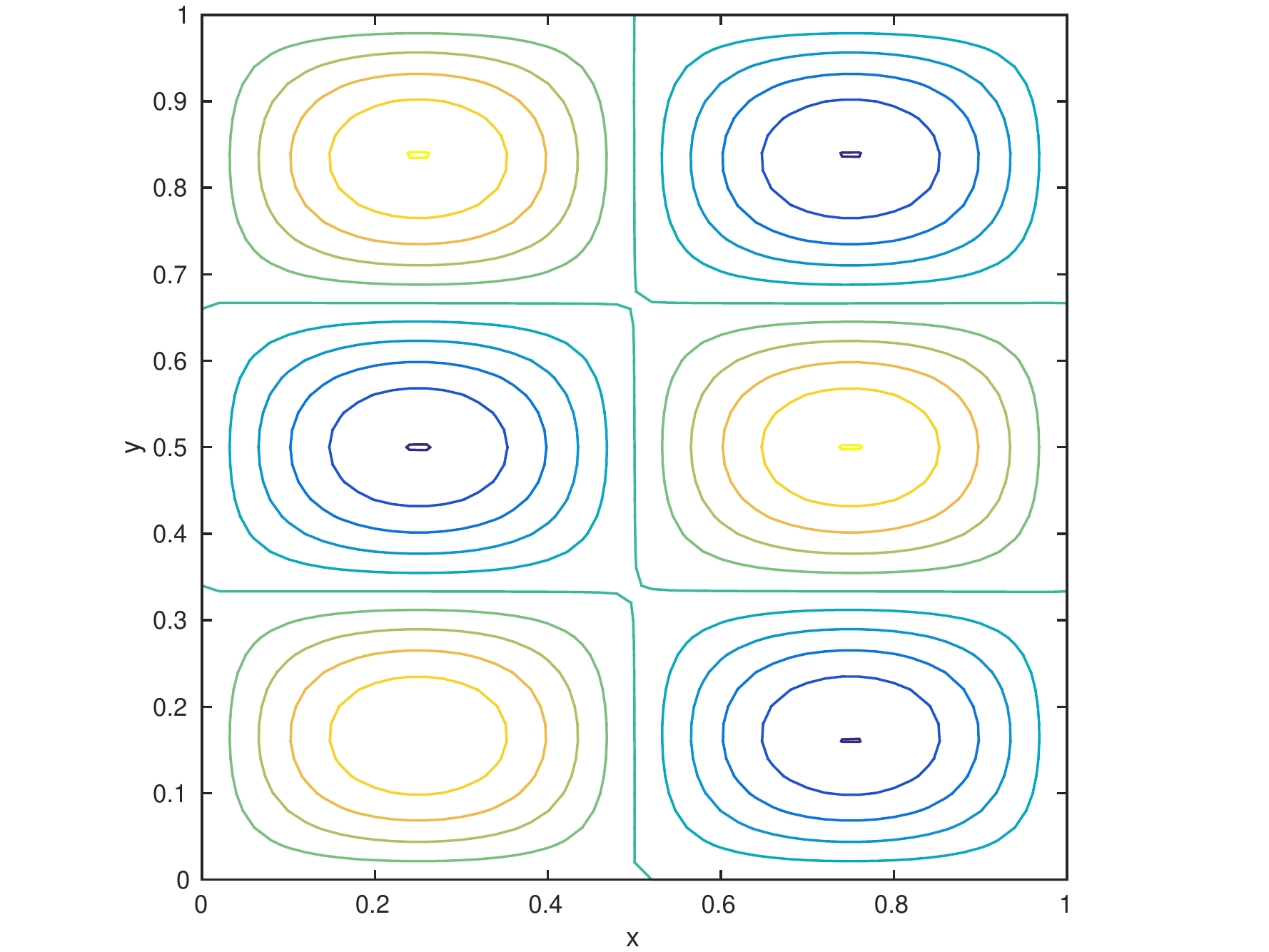}
                \caption{Numerical solution with $J=51$}
                \label{exam3-f1b}
        \end{subfigure}
	\caption{Example \ref{exam3}. Contours of the exact and computed solutions.}
	\label{exam3-f1}
\end{figure}

\begin{figure}
        \centering
        \begin{subfigure}[b]{0.5\linewidth}
                \includegraphics[width=\textwidth]{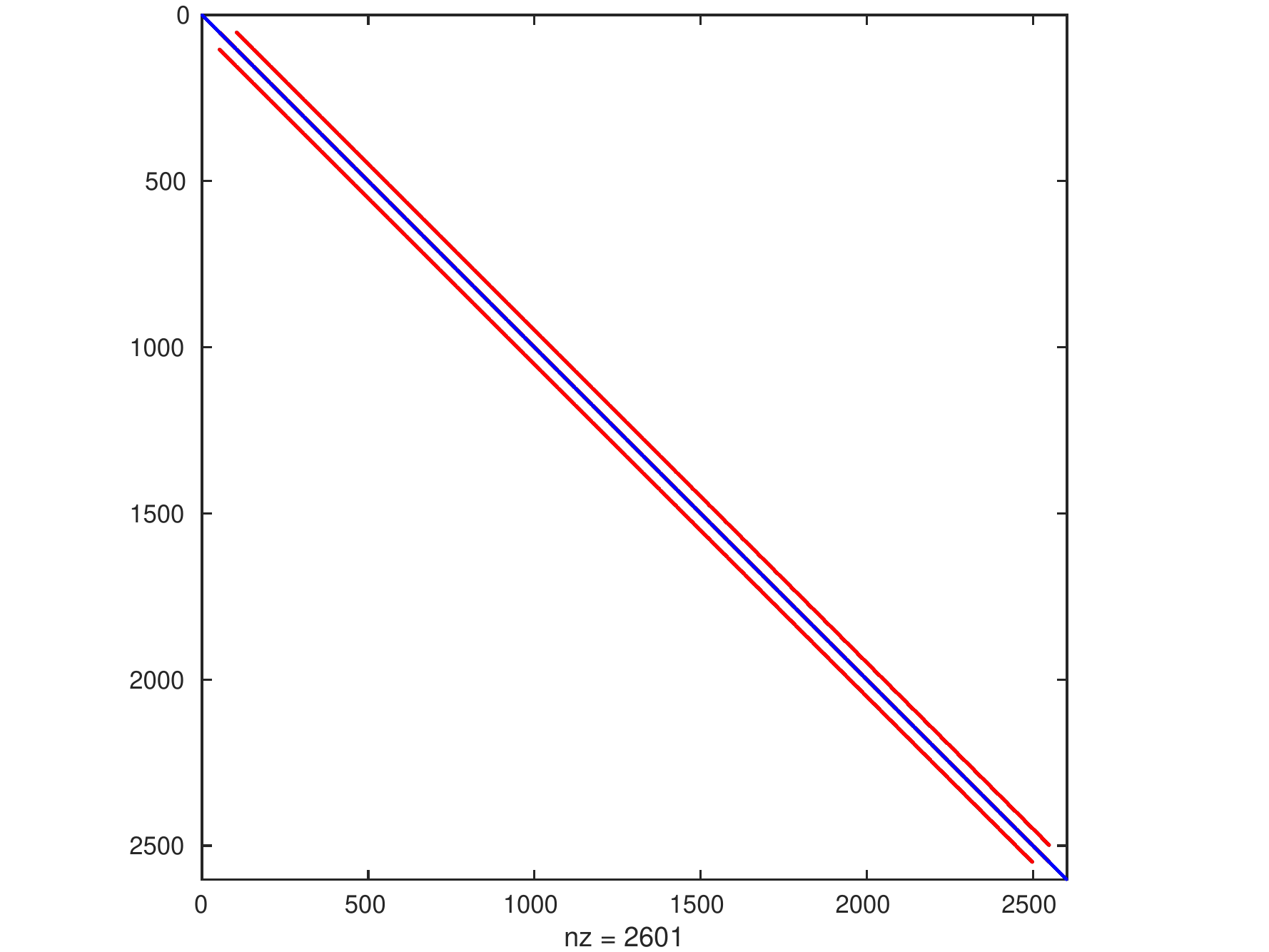}
                \caption{Sign pattern for J = 51}
                \label{exam3-f2a}
        \end{subfigure}%
        \begin{subfigure}[b]{0.5\linewidth}
                \includegraphics[width=\textwidth]{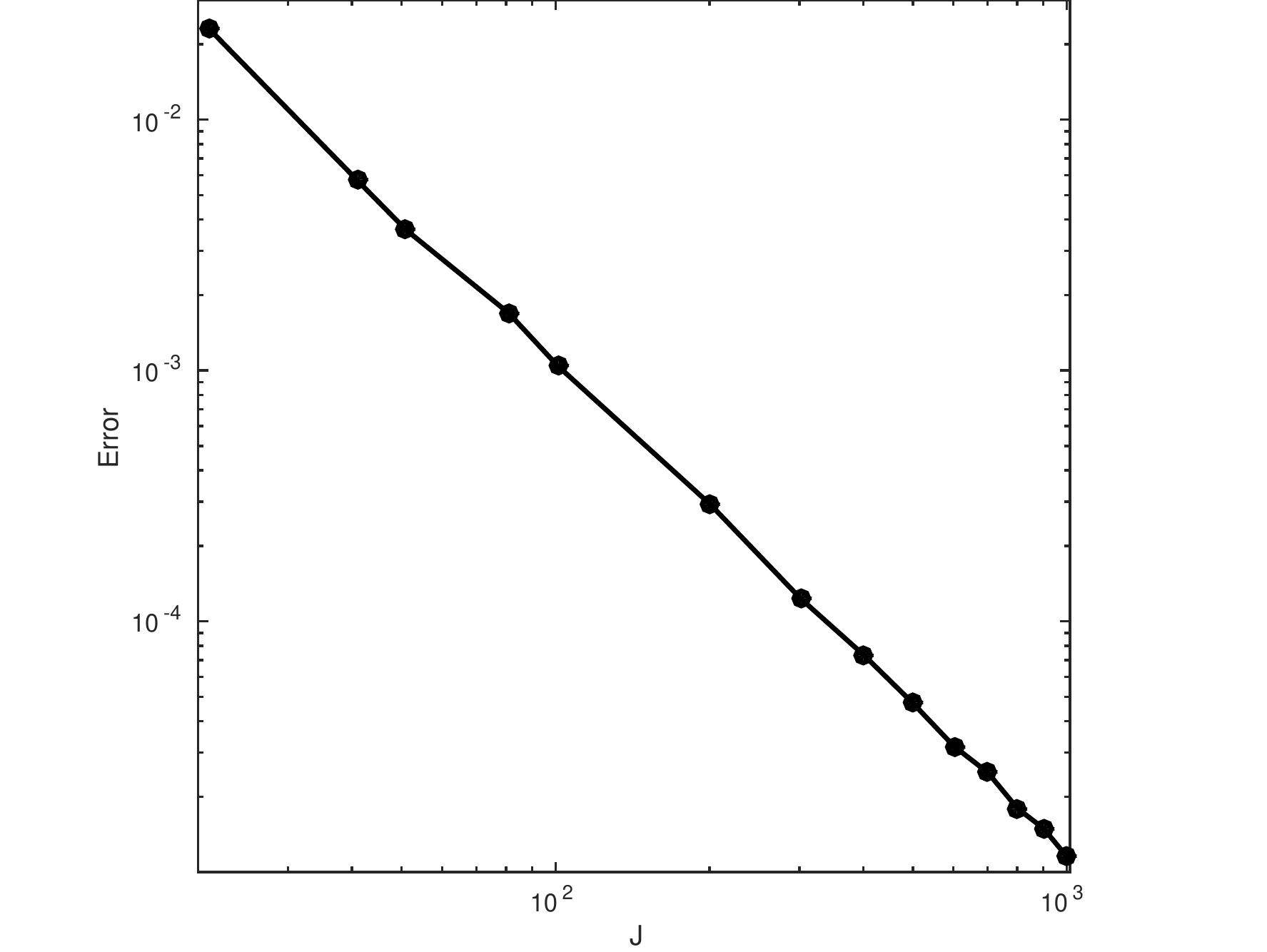}
                \caption{Convergence history}
                \label{exam3-f2b}
        \end{subfigure}
	\caption{Example \ref{exam3}. Sign pattern of the coefficient matrix and the convergence history for
		the maximum norm of the error of the FD solution.}
	\label{exam3-f2}
\end{figure}
\end{exam}

\begin{exam}
\label{exam4}
In this final example, the diffusion matrix is given by
\begin{equation}
\label{exam4-D}
\mathbb{D} = \begin{bmatrix}
			\cos \theta & -\sin \theta \\ 
			\sin \theta & \cos \theta
		\end{bmatrix} \begin{bmatrix}
			k & 0 \\ 
			0 & 1 
		\end{bmatrix} \begin{bmatrix}
			\cos \theta & \sin \theta \\ 
			-\sin \theta & \cos \theta
		\end{bmatrix} ,
\end{equation}
where $\theta = \pi \sin(x) \cos(y)$ and $k$ is a parameter. The bigger $k$ is, the more anisotropic the diffusion matrix is.
The more anisotropic $D$ is, the larger stencil is required for the construction of a monotone FD scheme. We consider
two cases with $k = 10$ and $k = 100$. For both of these cases, we again choose $f$ and $g$ in \eqref{had} such that the exact solution is $u(x,y) = \sin(2 \pi x) \sin(3 \pi y)$. 

The numerical calculation suggests that a $7 \times 7$ stencil be needed
to construct a monotone scheme for $k=10$ and a stencil of $53 \times 53$ for the case $k=100$.
The second order convergence can be seen in Fig.~\ref{exam4-f2} for both cases although the error is
almost two magnitude larger for the much more anisotropic case with $k=100$.


\begin{figure}
\begin{center}
\includegraphics[scale=0.5]{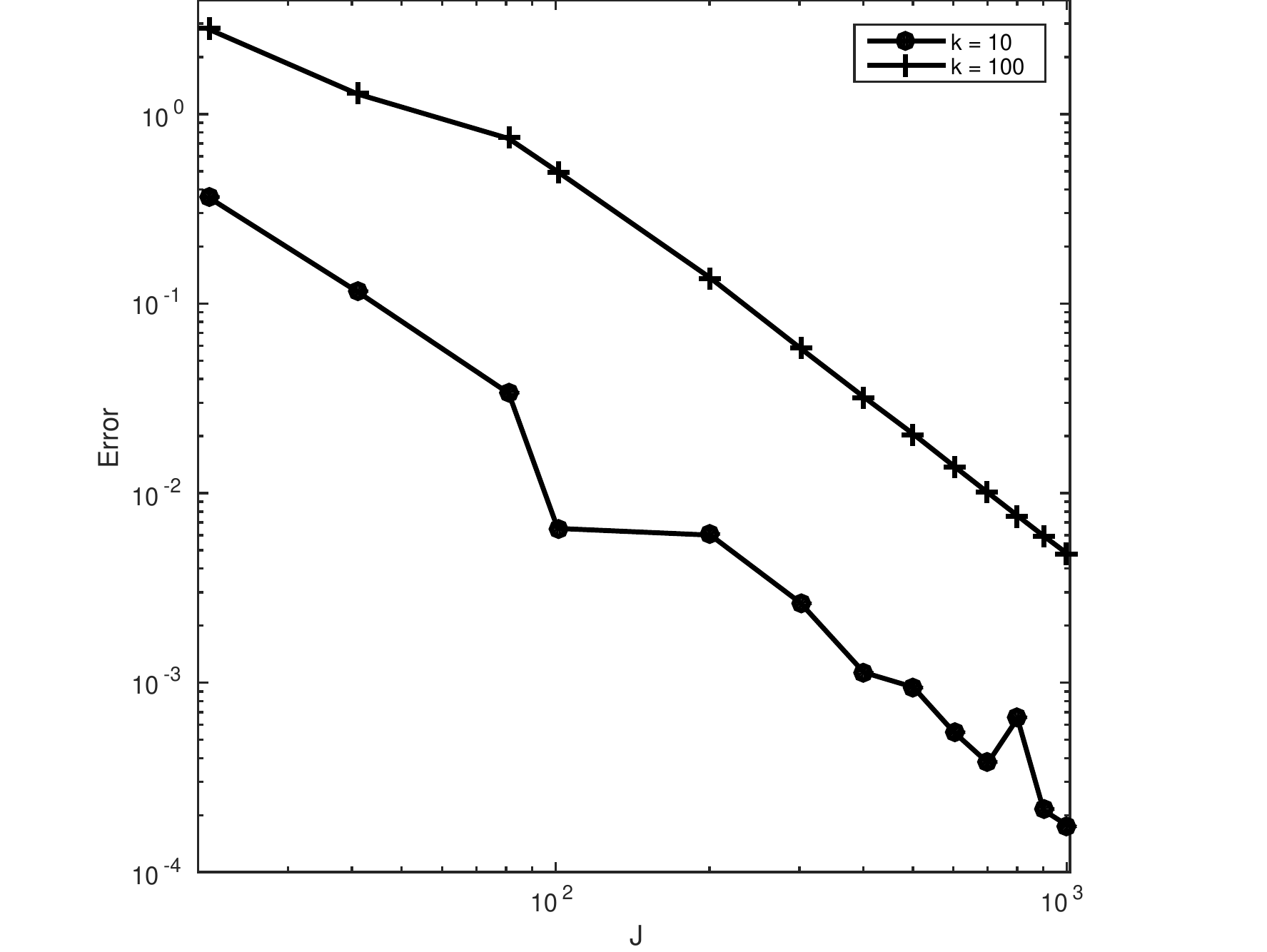}
\caption{Example \ref{exam4}. The convergence history for the maximum norm of the error of the FD solution.}
\label{exam4-f2}
\end{center}
\end{figure}

\end{exam}

\section{Conclusions}
\label{sec:conclude}

In the previous sections we have studied the construction of monotone finite difference schemes
for the divergence form (\ref{had}) using nonnegative directional splittings. The main results are stated in Theorems~\ref{thm1},
\ref{thm3.1}, and \ref{thm3.2}. Theorem~\ref{thm1} states that the diffusion operator in (\ref{had})
has a nonnegative directional splitting at any arbitrary interior point (cf. (\ref{splitting-1}))
when the diffusion matrix $\mathbb{D}$ is continuous
on $\overline{\Omega}$ and symmetric and uniformly positive definite in $\Omega$. This splitting
holds in a neighborhood of the point in the sense that the corresponding coefficients are nonnegative
in the neighborhood. Moreover, the size of the neighborhood depends only on $D$. This result
is an important improvement over that in \cite{Weickert} where the coefficients of the splitting
are nonnegative only at a point. Another improvement is that the function $b(x,y)$ in (\ref{eqn:D})
is allowed to change sign on $\Omega$. These improvements are crucial to the construction of monotone
finite difference schemes for (\ref{had}) since any finite difference discretization for the divergence form
(\ref{had}) requires information of the coefficients in a neighborhood of a mesh point.

The construction of monotone finite difference schemes using the nonnegative directional splitting (\ref{splitting-1})
has been discussed in \S\ref{sec:FD}. A key to the construction is to define principal directions for the underlying
finite difference stencil \cite{Weickert}. We have shown that some of these principal directions can be used
as the splitting directions when the mesh spacing is sufficiently small and the size of finite difference stencil
is sufficiently large. Note that the stencil size has a uniform bound determined by $\mathbb{D}$; see (\ref{m-1}).
This bound can become large when $\mathbb{D}$ is strongly anisotropic. In that case, not only
a large stencil but a very fine mesh have to be used. In such a situation, the constructed finite difference
scheme may be impractical and a more sophisticated way to construct monotone finite difference schemes
may be needed, which is a research topic worth more investigations in the future.
On the other hand, Theorem~\ref{thm3.2} shows that a $3\times 3$ stencil can be used
to construct monotone finite difference schemes when $\mathbb{D}$ is strictly diagonally dominant.
Numerical results in \S\ref{sec:num} are consistent with the theoretical findings.

\end{document}